\documentclass[a4paper,
10pt,
pdftex,
normalheadings,
headsepline,
footsepline,
onecolumn,
headinclude,
footinclude,
DIV14,
abstracton]
{scrartcl}

\usepackage
{
	graphicx,
	amssymb,
	amsmath,
	amsthm,
	xcolor,
	dsfont,
	algpseudocode,
	authblk,
}

\usepackage[ngerman, english]{babel}

\usepackage[
left=25mm,
right=25mm,
top=25mm,
bottom=30mm
]{geometry}

\usepackage[bf]{caption}
\usepackage[colorlinks,
pdffitwindow=false,
plainpages=false,
pdfpagelabels=true,
pdfpagemode=UseOutlines,
pdfpagelayout=SinglePage,
bookmarks=false,
colorlinks=true,
hyperfootnotes=false,
linkcolor=blue,
citecolor=green!50!black]
{hyperref}

\usepackage{enumerate}
\usepackage{algorithm}

\usepackage{tabularx}

\newcommand{\R}{\mathbb{R}}

\newcommand{\Kcal}{\mathcal{K}}
\newcommand{\Ucal}{\mathcal{U}}

\newcommand{\Xcal}{\mathcal{X}}
\newcommand{\Vcal}{\mathcal{V}}

\newcommand{\bzero}{\mathbf{0}}
\newcommand{\bu}{\mathbf{u}}
\newcommand{\bubar}{\mathbf{\bar{u}}}
\newcommand{\ubar}{\bar{u}}
\newcommand{\bx}{\mathbf{x}}
\newcommand{\bxi}{\boldsymbol{\xi}}
\newcommand{\bxdot}{\dot{\mathbf{x}}}
\newcommand{\bX}{\mathbf{X}}
\newcommand{\by}{\mathbf{y}}
\newcommand{\bv}{\mathbf{v}}
\newcommand{\bz}{\mathbf{z}}
\newcommand{\bB}{\mathbf{B}}
\newcommand{\bF}{\mathbf{F}}
\newcommand{\bG}{\mathbf{G}}
\newcommand{\bH}{\mathbf{H}}
\newcommand{\bK}{\mathbf{K}}
\newcommand{\bPhi}{\boldsymbol{\Phi}}
\newcommand{\bPsi}{\boldsymbol{\Psi}}
\newcommand{\bPsidot}{\boldsymbol{\dot{\Psi}}}
\newcommand{\bnabla}{\boldsymbol{\nabla}}

\newtheorem{theorem}{Theorem}[section]

\newtheorem{remark}[theorem]{Remark}
\newtheorem{example}[theorem]{Example}

\let\OLDthebibliography\thebibliography
\renewcommand\thebibliography[1]{
	\OLDthebibliography{#1}
	\setlength{\parskip}{0pt}
	\setlength{\itemsep}{0pt plus 0.3ex}
}

\date{}

\begin{document}
	
	\title{Data-Driven Model Predictive Control using Interpolated Koopman Generators}
	\author[1]{Sebastian Peitz}
	\author[2]{Samuel E.~Otto}
	\author[2]{Clarence W.~Rowley}
	\affil[1]{\normalsize Department of Mathematics, Paderborn University, Germany}
	\affil[2]{\normalsize Department of Mechanical and Aerospace Engineering, Princeton University, USA}
		
	\maketitle
	
	\begin{abstract}
		In recent years, the success of the Koopman operator in dynamical systems analysis has also fueled the development of Koopman operator-based control frameworks.
		In order to preserve the relatively low data requirements for an approximation via Dynamic Mode Decomposition, a quantization approach was recently proposed in \cite{PK19}. 
		This way, control of nonlinear dynamical systems can be realized by means of switched systems techniques, using only a finite set of autonomous Koopman operator-based reduced models. 
		These individual systems can be approximated very efficiently from data. The main idea is to transform a control system into a set of autonomous systems for which the optimal switching sequence has to be computed. 
		In this article, we extend these results to continuous control inputs using relaxation.
		This way, we combine the advantages of the data efficiency of approximating a finite set of autonomous systems with continuous controls. 
		We show that when using the Koopman generator, this relaxation --- realized by linear interpolation between two operators --- does not introduce any error for control affine systems. 
		This allows us to control high-dimensional nonlinear systems using bilinear, low-dimensional surrogate models. 
		The efficiency of the proposed approach is demonstrated using several examples with increasing complexity, from the Duffing oscillator to the chaotic fluidic pinball.
	\end{abstract}

	\section{Introduction}
	\label{sec:Introduction}
	
	Real-time control of high-dimensional dynamical systems is a very challenging task, in particular for nonlinear systems; see, e.g., \cite{BN15} for a recent survey on turbulent flow control. 
	To this end, advanced control techniques such as \emph{Model Predictive Control (MPC)} \cite{GP17} or machine learning-based control \cite{DBN17} have gained more and more attention in recent years. 
	In MPC, an open-loop optimal control is computed repeatedly on a finite-time horizon using a model of the system dynamics. 
	This results in a feedback loop, but the open-loop problem has to be solved in a very short time, which is in general infeasible for complex systems such as nonlinear partial differential equations (PDEs), at least when using classical discretization approaches such as finite elements.
	
	A remedy to this problem is to use a surrogate model which can be solved significantly faster; see \cite{LMQR14,BGW15} for overviews. 
	Besides classical approaches such as \emph{Proper Orthogonal Decomposition (POD)} \cite{Sir87,KV99,Row05,HV05,BDPV18}, an approach that has attracted a lot of attention in recent years is to construct a \emph{reduced order model (ROM)} by means of the \emph{Koopman operator} \cite{Koo31}. 
	This is an infinite-dimensional but linear operator describing the dynamics of observables, and the approach can even be applied to sensor measurements or in situations where the underlying system dynamics are unknown. 
	Significant advances have been made over the past years both theoretically (see, e.g., \cite{Mez05,BMM12}) as well as numerically. 
	In the latter case, \emph{Dynamic Mode Decomposition (DMD)} \cite{Sch10,RMB+09,TRL+14,KGPS18} or \emph{Extended Dynamic Mode Decomposition (EDMD)} \cite{WKR15,KKS16,KNP+19} are the most popular algorithms. 
	
	More recently, various attempts have been made to use Koopman operator-based ROMs for control \cite{PBK15,PBK18,BBPK16,KM18b,KKB17,Williams2016extending}. 
	In many of these approaches, the Koopman operator is approximated for an augmented state (consisting of the actual state and the control) in order to deal with the non-autonomous control system. 
	Alternatively, the approach presented in \cite{Williams2016extending} treats the control like a time-varying parameter in order to learn EDMD-like models that interpolate Koopman operator approximations over the range of input.
	Since the input affects the Koopman operator and its eigenfunctions in a nonlinear way, it is necessary to include many nonlinear functions of the input in addition to the state variables in all of the above methods.
	For this reason, large amounts of data are required to cover a sufficiently large range of the dynamics. 
	An alternative approach has been presented in \cite{PK19}, where the control system is replaced by a set of autonomous systems with constant inputs. 
	This way, the optimal control problem is transformed into a switching problem between linear systems in a nonlinear feature space.
	If enough nonlinear features are taken, then the original objective can be approximated arbitrarily closely \cite{KM18a}. 
	The major advantage of this approach is that the required amount of data is very small \cite{PK20}.
	However, a drawback is that the resulting control problem is of combinatorial nature and that the input is restricted to a finite set.
	
	In this article, we begin to examine how the Koopman operators and generators can be properly interpolated in the sense of \cite{Williams2016extending} to yield efficient data-driven reduced-order models for model predictive control (MPC).
	When the dynamics are control-affine, we show that the Koopman generators are affine with respect to the input (which was also observed in \cite{Goswami2017global} in a similar setting), thereby justifying the use of affine interpolation between models based on the Koopman generator.
	This affine property is advantageous since it allows us to construct bilinear reduced order models using a more data-efficient version of the method in \cite{Williams2016extending} to approximate the Koopman generators.
	Furthermore, we show that input-affine interpolation of finite-time Koopman operators is accurate to first-order in time, yielding bilinear discrete-time models that can efficiently predict dynamics with short zero-order holds on the input.
	
	These efficient models in continuous and discrete time enable us to extend the MPC ideas of \cite{PK19, PK20} to utilize continuous inputs, rather than switching between fixed input levels, without increasing the training data requirements.
	In addition, we note that the affine interpolation methods we propose are closely related to the relaxation approaches used to solve the switching problem in mixed-integer optimal control \cite{Sag09}.
	When the control enters in a nonlinear way, our approach can therefore be viewed as a relaxation of the switching MPC approaches \cite{PK19, PK20}, allowing for intermediate control values.
	Embedded in an MPC framework, we show that the affine interpolation approaches yields remarkable results for several nonlinear control systems. 
	As examples, we consider several systems of increasing complexity, from the Duffing oscillator over the Burgers equation to the \emph{fluidic pinball} \cite{DPMN18}, a fluid flow problem with chaotic behavior governed by the 2D incompressible Navier--Stokes equations.
	
	The remainder of the article is structured as follows. 
	After introducing the control framework and the switched systems approach using Koopman operator-based models in Sec.~\ref{sec:Koopman}, we present the relaxation extension for control affine systems in Sec.~\ref{sec:Continuous}. 
	The incorporation into an MPC framework is discussed in Sec.~\ref{sec:MPCsolution}, and numerical results are then presented in Sec.~\ref{sec:MPC}.
	
	\section{Koopman operator-based optimal control}
	\label{sec:Koopman}
	
	The overall goal we pursue is to efficiently solve nonlinear optimal
        control problems (OCPs) such as the following:
	\begin{equation}\label{eq:MPC}
		\begin{aligned}
			\min_{\bu \in L^2([t_0,t_e], \Ucal)} &\int_{t_0}^{t_e} L(\bx(t),\bu(t),t) \ dt \\
			\mbox{s.t.} \quad \bxdot(t) &= \bH(\bx(t), \bu(t)), \\ 
			\bx(t_0) &= \bx_0,
		\end{aligned}
	\end{equation}
	where $\bx(t) \in \Xcal \subseteq \R^n$ is the state and $\bu(t) \in \Ucal$ is the control. The system dynamics are defined by $\bH \colon \Xcal \times \Ucal \rightarrow \Xcal$.
	The system dynamics are formulated as an ordinary differential equation, but partial differential equations can be considered analogously.
	
	For real systems, it is often insufficient to determine open-loop control inputs. 
	A remedy to this issue is MPC \cite{GP17}, where the OCP is solved repeatedly in real-time over a \emph{prediction horizon} of finite length $t_p=t_e-t_0$, cf., e.g., \cite{GP17}. 
	A (small) part $\Delta t \leq t_p$ of the solution of \eqref{eq:MPC} is then applied to the real system while the optimization is repeated with the prediction horizon moving forward by $t_c$.
	
	The main challenge in MPC is that Problem~\eqref{eq:MPC} has to be solved within the time $\Delta t$ which can be very challenging for high-dimensional nonlinear systems, and additional measures have to be taken in order to achieve real-time applicability. 
	One such approach is via the Koopman operator, which will be described next.
	
	\subsection{Koopman operator and EDMD}
	\label{subsec:Koopman_operator_and_EDMD}
	For the moment, assume that $\bu = \bzero$; i.e., we consider the autonomous system $ \bxdot = \bH_{\bzero}(\bx) = \bH(\bx, \bzero) $. 
	By $\psi \colon \Xcal \rightarrow \R$ we define a real-valued observable of the system.
	Then, the so-called \emph{Koopman semigroup} of operators $ \{ \Kcal^{\Delta t} \}_{\Delta t \geq 0} $ is defined as
	\begin{equation*}
		(\Kcal^{\Delta t} \psi)(\bx) = \psi(\bPhi^{\Delta t}(\bx)),
	\end{equation*}
	see \cite{LaMa94, BMM12, KKS16}. Here,
	\[
		\bx(t+\Delta t) = \bPhi^{\Delta t}(\bx(t))
	\]
	is the flow of the dynamical system.
	The Koopman operator $\Kcal^{\Delta t}$ is linear but usually acts on an infinite-dimensional space of observables.
	
	The most popular approach to numerically approximate the Koopman operator is \emph{Dynamic Mode Decomposition (DMD)} \cite{RMB+09,Sch10}, which was later extended to nonlinear observables in the \emph{Extended Dynamic Mode Decomposition (EDMD)} \cite{WKR15,KKS16,KNP+19}. 
	It is a modal decomposition method for large data sets such as fluid flow simulations. While being similar to Proper Orthogonal Decomposition, the main difference is that instead of obtaining a basis with minimal $L^2$ projection error, each of the DMD modes possesses a frequency with which it oscillates, determined by the corresponding complex eigenvalue \cite{RMB+09}.
	
	EDMD constructs an approximation of the Koopman operator from data (i.e., measurements) given by $ \bz = \boldsymbol{\psi}(\bx) \in \R^{n_o}$. 
	For finite-dimensional systems, it is also possible to observe the entire (discretized) state (\emph{full state observable}). In order to obtain a larger \emph{feature space}, these observations are often \emph{lifted} using, for instance, monomials, Hermite polynomials or radial basis functions.
	For a given set of observables $ \{ \psi_{1},\,\psi_{2},\,\dots,\,\psi_{n_o} \} $, we define $ \boldsymbol{\psi} \colon \mathcal{X} \to \R^{n_o} $ by
	\begin{equation*}
		\boldsymbol{\psi}(\bx) =
		\begin{bmatrix}
		\psi_{1}(\bx) & \psi_{2}(\bx) & \dots & \psi_{n_o}(\bx)
		\end{bmatrix}^{\top}.
	\end{equation*}
	Note that for $ \boldsymbol{\psi}(\bx) = \bx $, we recover the standard DMD. 
	The available measurement or simulation data (either obtained from one long or multiple short time series) can be used to compute the values of the observables at pairs of states $\mathbf{x}_i$ and $\widetilde{\bx}_i = \bPhi^{\Delta t}(\bx_i)$ separated in time by $\Delta t$.
	Arranging these data in matrices
	\begin{equation*}
		\begin{aligned}
			\bPsi_{\bX} &=
			\begin{bmatrix} \boldsymbol{\psi}(\bx_{1}) & \boldsymbol{\psi}(\bx_{2}) & \dots & \boldsymbol{\psi}(\bx_{m}) \end{bmatrix}
			~ \text{and} ~
			\bPsi_{\widetilde{\bX}} &=
			\begin{bmatrix} \boldsymbol{\psi}(\widetilde{\bx}_{1}) & \boldsymbol{\psi}(\widetilde{\bx}_{2}) & \dots & \boldsymbol{\psi}(\widetilde{\bx}_{m}) \end{bmatrix}
		\end{aligned}
	\end{equation*}
	allows one to compute a matrix approximation $ \bK^{\Delta t} \in \R^{k \times k} $ of the Koopman operator $\mathcal{K}^{\Delta t}$ in the span of the chosen observables given by
	\begin{equation*}
		\bK^{\Delta t} = \bPsi_{\widetilde{\bX}} \bPsi_{\bX}^+ = \big( \bPsi_{\widetilde{\bX}} \bPsi_{\bX}^{\top} \big) \big(\bPsi_{\bX} \bPsi_{\bX}^{\top}\big)^+,
    \end{equation*}
    where $()^+$ denotes the Moore-Penrose pseudoinverse of a matrix.
	
	\subsection{Switched system control}
	\label{subsec:Koopman_MPC}
	As mentioned at the beginning of this Section, we want to solve Problem~\eqref{eq:MPC} in real-time.
	According to the approach in \cite{PK19}, we realize this by
	taking two steps:
	\begin{enumerate}[i)]
		\item replace the dynamical control system in \eqref{eq:MPC} by a finite number of autonomous systems
			\[\bH_{\bubar}(\bx) = \bH(\bx,\bubar)\]
			with constant input $\bubar \in \hat{\Ucal} = \{\bubar_1,\ldots,\bubar_{n_c}\}$;
		\item construct linear systems for low-dimensional observations of the $\bH_{\bubar_j}$ using the Koopman operator.
	\end{enumerate}
	In step i), Problem \eqref{eq:MPC} is transformed into a switching problem, where the objective is to select the optimal right hand side in each time step:
	\begin{equation}\label{eq:MPC_STO}
		\begin{aligned}
			\min_{\bu:\ \bu(t)\in\hat{\Ucal}} \int_{t_0}^{t_e} & L(\bx(t),\bu(t), t) \ dt \\
			\mbox{s.t.} \quad \bxdot(t) &= \bH_{\bu(t)}(\bx(t)), \\
			\bx(t_0) &= \bx_0,
		\end{aligned}
	\end{equation}
	which differs from~\eqref{eq:MPC} only in that $\Ucal$ is replaced by $\hat{\Ucal}$, and the dynamics is now governed by~$\bH_{\bu(t)}$.
	In other words, $\bu(t)$ describes which system $\bH_{\bubar_j}$, $j\in\{1, \ldots, n_c\}$ to apply at time $t$. 
	Problem~\eqref{eq:MPC_STO} is a mixed integer optimal control problem (MIOCP) which can be solved using relaxation and rounding techniques; see, e.g., \cite{SBR09}.
	
	Solving either \eqref{eq:MPC} or \eqref{eq:MPC_STO} numerically can quickly become very expensive for large scale systems such that real-time applicability is infeasible. 
	Furthermore, there are many systems where the dynamics is not explicitly known. 
	In both situations, we can use observations (i.e., data) to approximate the Koopman operator and derive a linear system describing the dynamics of these observations. 
	These observations could consist of (part of) the system state as well as arbitrary functions of the state such as the lift coefficient of an object within a flow field.
	
	Following step ii), we would like to use a \emph{Koopman operator-based reduced order model (K-ROM)} to overcome the issue of large computational cost. 
	To this end, we compute $n_c$ Koopman operators for the $n_c$ different autonomous systems:
	\begin{equation*}
	(\Kcal^{\Delta t}_{\bubar_j} \psi)(\bx) = \psi(\bPhi^{\Delta t}_{\bubar_j}(\bx)),\quad j = 1,\ldots,n_c,
	\end{equation*}
	where $\Phi^{\Delta t}_{\bubar_j}$ is now the family of flow maps corresponding to the fixed control inputs $\bubar_j$.
	Using EDMD, we can compute an approximation of the individual Koopman
        operators from observations of the respective systems, and introducing
        the reduced state $\bz = \boldsymbol{\psi}(\bx)$, we obtain the
        following linear systems:
	\begin{equation}\label{eq:DiscreteKoopmanDynamics}
	\begin{aligned}
	\bz_{i+1} &= \bK^{\Delta t}_{\bubar_j} \bz_i,\quad j = 1,\ldots,n_c \\
	\bz_0 &= \boldsymbol{\psi}(\bx_0).
	\end{aligned}
	\end{equation} 
	If a zero-order hold is placed on the input over $\Delta t$, these linear dynamics now replace the original dynamics in \eqref{eq:MPC_STO}, yielding the following K-ROM-based OCP over the prediction horizon of length $\ell\Delta t$:
	\begin{equation} \label{eq:MPC_STO_Koopman}
		\begin{aligned}
			\min_{\bu_0,\ldots,\bu_{\ell-1} \in \hat{\Ucal}} \sum_{i=0}^{\ell-1} &\hat{L}_{i+1}(\bz_{i+1},\bu_{i+1}) \\
			\mbox{s.t.}\quad \bz_{i+1} &= \bK^{\Delta t}_{\bu_{i}} \bz_{i} \quad \text{for}~i = 0,\ldots,\ell-1, \\
			\bz_0 &= \boldsymbol{\psi}(\bx_0). 
		\end{aligned}
	\end{equation}
	Note that we have introduced a discrete-time formulation of problem \eqref{eq:MPC_STO} as the Koopman operator yields a discrete-time system. 
	Furthermore, $\hat{L}_i$ is an expression of the objective function value in terms of the measurements $\bz$ at time $i\Delta t$.
	
	The key advantage over other approaches where one operator is computed for an augmented state $\hat{\bx} = (\bx,\bu)$ (see, e.g., \cite{PBK15,PBK18,KM18b}) is that the individual models \eqref{eq:DiscreteKoopmanDynamics} can be approximated with very low data requirements. 
	It is often sufficient to use fewer than 100 data points for each system \cite{PK20}. 
	
	\section{Continuous control inputs}
	\label{sec:Continuous}
	
	Using the switched systems approach presented in Section~\ref{subsec:Koopman_MPC}, PDE-constrained control problems can be solved several orders of magnitude faster when replacing the PDE constraint by the K-ROM \cite{PK19}.
	However, several drawbacks occur:
	\begin{enumerate}[i)]
		\item the optimization problem is of combinatorial nature, and is thereby harder to solve;
		\item the control input is limited to a small number of values: i.e., to the finite set $\hat{\Ucal} \subset \Ucal$;
		\item the discrete-time setting additionally limits the control flexibility; whereas large lag times are often beneficial for the numerical approximation of the Koopman operator, they reduce the control frequency.
	\end{enumerate}
	In order to overcome these drawbacks, we shall instead build our K-ROMs from approximations of the generators of the Koopman semigroups under piecewise constant input.
	For control-affine systems, we shall find that the Koopman generators also have an affine property with respect to the input, enabling continuous interpolation between models at different input levels.
	
	\subsection{Koopman generator-based models for control}
	Consider the case of control-affine dynamics $\bH$ in \eqref{eq:MPC}:
	\begin{equation*}
	\bxdot = \bH(\bx,\bu) = \bF(\bx) + \bG(\bx)\bu.
	\end{equation*}
	If a constant input level $\bu(t) \equiv \bubar\in\Ucal$ is supplied, then the system induces a family of flows $\bx(t+\Delta t) = \bPhi_{\bubar}^{\Delta t}(\bx(t))$, on the state space parameterized by the input level $\bubar$.
	Supposing that we can find a function space $\Vcal$ so that $\psi\circ
        \bPhi_{\bubar}^{\Delta t}\in\Vcal$ for every element $\psi\in\Vcal$,
        time interval $\Delta t\ge 0$, and input level $\bubar$, then a family of Koopman operators
	\begin{equation*}
	\Kcal_{\bubar}^{\Delta t} \psi \triangleq \psi\circ \bPhi_{\bubar}^{\Delta t}
	\end{equation*}
	can be defined on $\Vcal$.
	
	Observe that at each fixed input level $\bubar$, the family of Koopman operators
	$\lbrace \Kcal_{\bubar}^{\Delta t} \rbrace_{\Delta t \ge 0}$ is a semi-group generated by the operator $\Kcal_{\bubar}:\Vcal\to\Vcal$ defined according to
	\begin{equation*}
	\Kcal_{\bubar} \psi \triangleq \lim_{\Delta t\to 0^+} \frac{\psi\circ \bPhi_{\bubar}^{\Delta t} - \psi}{\Delta t},
	\end{equation*}
	if the limit exists.
	Recall the well-known \cite{van2014probability, liggett2010continuous, Kolmogoroff1931} relationship between the Koopman generator and Koopman operator stated in the following theorem (\ref{thm:KoopmanGenerator}), originally due to Kolmogorov in the stochastic setting.
	Intuitively, the Koopman generator applied to an observable $\psi$ returns the new observable $\dot{\psi}_{\mathbf{\bar{u}}}$ which is the time derivative of $\psi$ given by the dynamics.
	The finite time evolution of an observable is then recovered by integrating these time derivatives to produce a trajectory in the space of observables.
	The corresponding flow on this space is described by the exponential map of the Koopman generator.
	\begin{theorem}[Koopman-Kolmogorov]
		\label{thm:KoopmanGenerator}
		Given a constant input $\bu(t) \equiv \bubar$, the system is autonomous.
		In this setting, the family of Koopman operators parameterized by time is related to the Koopman generator by
		\begin{equation}
		\frac{d}{d t} \Kcal_{\bubar}^t = \Kcal_{\bubar}^t \Kcal_{\bubar},\quad
		\forall t\in [0, \epsilon), \quad
		\Kcal_{\bubar}^0 = \mathcal{I},
		\label{eqn:KoopmanOperatorODE}
		\end{equation}
		yielding the identity
		\begin{equation}
		\Kcal_{\bubar}^t = \exp (\Kcal_{\bubar} t) \triangleq \sum_{p=0}^{\infty} \frac{t^p}{p!} (\Kcal_{\bubar})^p,
		\label{eqn:KoopmanExponential}
		\end{equation}
		where $(\Kcal_{\bubar})^p$ denotes repeated application of $\Kcal_{\bubar}$ and $(\Kcal_{\bubar})^0 = \mathcal{I}$ is the identity operator.
	\end{theorem}
	\begin{proof}
		See appendix \ref{app:KoopmanKolmogorovProof} for a simple and illustrative proof in the deterministic setting.
	\end{proof}
	
	If $\Vcal$ is a space of differentiable functions, then the action of the Koopman generator is described by the differential operator
	\begin{equation}\label{eq:GeneratorPDE}
	\Kcal_{\bubar} \psi
	= \bF \cdot \bnabla_{\bx}\psi + (\bG \bubar)\cdot\bnabla_{\bx}\psi
	\end{equation}
	(see, e.g., \cite{KKB17,KNP+19}).
	Defining $\mathcal{B}_{\bubar} \triangleq \Kcal_{\bubar} -
        \Kcal_{\bzero}$, we obtain
	\begin{equation*}
	\mathcal{B}_{\bubar}\psi = (\bG\bubar)\cdot\bnabla_{\bx}\psi.
	\end{equation*}
	Remarkably, these operators $\mathcal{B}_{\bubar}$ are linear with respect to the input~$\bubar$.
	Specifically, if $\bubar_1$ and $\bubar_2$ are inputs and $\alpha_1,\alpha_2\in\mathbb{R}$ then
	\begin{equation*}
	\mathcal{B}_{\alpha_1\bubar_1 + \alpha_2\bubar_2} = \alpha_1 \mathcal{B}_{\bubar_1} + \alpha_2 \mathcal{B}_{\bubar_2}.
	\end{equation*}
	Furthermore, this means that the Koopman generators are affine with
        respect to the input (cf.~also \cite{Goswami2017global}), as summarized
        in the following theorem.
	\begin{theorem}[Koopman generators are control-affine]
	\label{thm:GeneratorControlAffine}
	If $\Vcal$ is a space of differentiable functions and the dynamics are control-affine, then the Koopman generators are control-affine: that is,
	\begin{equation}
		\boxed{
			\Kcal_{\alpha_1\bubar_1 + \alpha_2\bubar_2} = \Kcal_{\bzero} + \alpha_1 \mathcal{B}_{\bubar_1} + \alpha_2 \mathcal{B}_{\bubar_2}.
		}
		\label{eqn:KoopmanGeneratorAffine}
	\end{equation}
	\end{theorem}
	Even though the Koopman {\em generators\/} are affine, the following example shows that the finite-time Koopman {\em operators\/} do not share this affine property: that is, in general,
	\begin{equation}
		\Kcal^{\Delta t}_{\alpha_1\bubar_1 + \alpha_2\bubar_2} 
		\neq \Kcal^{\Delta t}_{\bzero} + \alpha_1 (\Kcal^{\Delta t}_{\bubar_1} -  \Kcal^{\Delta t}_{\bzero}) + \alpha_2 (\Kcal^{\Delta t}_{\bubar_2} -  \Kcal^{\Delta t}_{\bzero}).
		\label{eqn:KoopmanAffine}
	\end{equation}
	
	\begin{example}[Affine Koopman counter-example]
		\label{eg:KoopmanCounterExample}
		Consider the control-affine dynamical system on the unit circle
		\begin{equation*}
		\dot{x} = u, \qquad x\in[0,2\pi],
		\end{equation*}
		whose constant-input flow map is given by $\Phi_{\bar{u}}^{\Delta t}(x) = x + u\Delta t \mod 2\pi$.
		The Koopman operator can then be defined over $\Vcal = L^2([0,2\pi])$.
		Furthermore, $\lbrace \psi_{k}: x\mapsto e^{\imath k x} \rbrace_{k\in\mathbb{Z}}$ is an orthogonal basis of eigenfunctions, since
		\begin{equation*}
			(\Kcal_{u}^{\Delta t}\psi_k)(x) = e^{\imath k (x + u\Delta t)} = e^{\imath k u\Delta t}\psi_k(x),
		\end{equation*}
		so the eigenvalues are $\lambda_k = e^{\imath k u\Delta t}$.
		Now let us check whether the input-affine property holds by applying the operators on both sides of Eq.~\eqref{eqn:KoopmanAffine} to $\psi_k$, giving
		\begin{equation*}
			\Kcal^{\Delta t}_{\alpha_1\ubar_1 + \alpha_2\ubar_2} \psi_k = e^{\imath k (\alpha_1\ubar_1 + \alpha_2\ubar_2)\Delta t} \psi_k,
		\end{equation*}
		and
		\begin{equation*}
			\Kcal^{\Delta t}_{0}\psi_k + \alpha_1 (\Kcal^{\Delta t}_{\ubar_1} -  \Kcal^{\Delta t}_{0})\psi_k + \alpha_2 (\Kcal^{\Delta t}_{\ubar_2} -  \Kcal^{\Delta t}_{0})\psi_k =
			\psi_k + \alpha_1 \left(e^{\imath k u_1 \Delta t} - 1 \right)\psi_k + \alpha_2 \left( e^{\imath k u_2 \Delta t} - 1 \right)\psi_k.
		\end{equation*}
		If the two operators act the same way on $\psi_k$, we must have
		\begin{equation*}
			e^{\imath k (\alpha_1\ubar_1 + \alpha_2\ubar_2)\Delta t} =
			1 + \alpha_1 \left(e^{\imath k u_1 \Delta t} - 1 \right) +
            \alpha_2 \left( e^{\imath k u_2 \Delta t} - 1
            \right),\qquad\forall \alpha_1,\alpha_2.
		\end{equation*}
		But this clearly does not hold for all $\alpha_1,\alpha_2$, as 
        one can verify by expanding the Taylor series for the exponentials.
		Interestingly, the above equation does hold to first order in $\Delta t$, as one might expect from the input-affine property \eqref{eqn:KoopmanGeneratorAffine} of the Koopman generators. 
		Furthermore, it is exact for linear systems with affine observables, as we shall see in Section~\ref{subsec:PracticalApproximation}.
    \end{example}
	
	The input-affine property of the Koopman generators expressed in Eq. \eqref{eqn:KoopmanGeneratorAffine} can be used to create a bilinear model for the time evolution of observations $\psi(\bx)$ in terms of a finite number, $n_c$, of Koopman generators.
	Given any linear combination of constant inputs to the system,
	\begin{equation*}
		\mathbf{u}(t) \equiv \mathbf{\bar{u}} = \sum_{i=1}^{n_c} \alpha_i\mathbf{\bar{u}}_i,
	\end{equation*}
	the dynamics of an observable are given by
	\begin{equation}\label{eqn:BilinearGeneratorModel}
		\dot{\psi}_{\mathbf{\bar{u}}}(\bx) = (\mathcal{K}_{\mathbf{\bar{u}}}\psi)(\bx) =  \left[\left(\Kcal_{\bzero} + \sum_{i=1}^{n_c}\alpha_i\mathcal{B}_{\bubar_i}\right)\psi\right](\bx).
	\end{equation}
	
	\begin{remark}[Time-Varying Input]
		\label{rem:TimeVaryingInput}
		The assumption that the input is constant may actually be relaxed.
		Suppose $\mathbf{x}(t)$ is a trajectory of the system under a time-varying input $\mathbf{u}(t)$ and that $\psi$ is a differentiable observable. 
		Then the time-derivative of the observable is still expressible as
		\begin{equation*}
			\frac{d}{dt}\psi(\mathbf{x}(t)) = \mathbf{H}(\mathbf{x}(t), \mathbf{u}(t))\cdot\boldsymbol{\nabla}_{\mathbf{x}}\psi(\mathbf{x}(t)) = (\mathcal{K}_{\mathbf{u}(t)}\psi)(\mathbf{x}(t)),
		\end{equation*}
		where $\mathcal{K}_{\mathbf{u}(t)}$ is the Koopman generator associated with constant input at the value $\mathbf{u}(t)$ at time $t$.
		Hence, $\mathcal{K}_{\mathbf{u}(t)}$ varies in time with $\mathbf{u}(t)$.
		The control-affine property allows us to pass the time-variation of the input onto the coefficients in an affine combination of fixed generators:
		expanding $\mathbf{u}(t) = \sum_{i=1}^{n_c} u_i(t)\mathbf{e}_i$ in the canonical basis, we obtain
		\begin{equation*}
			\mathcal{K}_{\mathbf{u}(t)} = \mathcal{K}_{\mathbf{0}} + \sum_{i=1}^{n_c}u_i(t)\mathcal{B}_{\mathbf{e}_i}.
		\end{equation*}
	\end{remark}
	
	\subsection{EDMD for generator-based modeling}
	\label{subsec:EDMDforGenerators}
	A numerical approximation of Eq.~\eqref{eqn:BilinearGeneratorModel} can be used to create reduced-order models for the time evolution of measurements $\bz$; cf.~\cite{KNP+19} for details.
	Similarly to the nonlinearly interpolated EDMD method of \cite{Williams2016extending}, we shall construct an approximation of the generator $\mathcal{K}_{\mathbf{u}}$ over a finite-dimensional subspace of observables spanned by the elements in a dictionary $\boldsymbol{\psi}$.
	The difference here is that we are approximating the {\em generator} and therefore need only consider {\em affine} interpolations, resulting in training data requirements no higher than those involved in approximating $n_c+1$ Koopman operators at fixed input levels separately.
	Recall that the Koopman operators for each autonomous system at fixed input levels were approximated separately in the switched optimal control formulation reviewed in section \ref{sec:Koopman} above.
	
	Without loss of generality, let $\mathbf{\bar{u}}_i = \mathbf{e}_i$, $i=1,\ldots,n_c$ be the canonical basis for the space of inputs $\mathcal{U}\subseteq\mathbb{R}^{n_c}$.
	This allows us to express Eq.~\eqref{eqn:BilinearGeneratorModel} directly in terms of the input components
	\begin{equation}\label{eqn:BilinearGeneratorModel_input}
		\dot{\psi}_{\mathbf{u}}(\bx) =  \left(\Kcal_{\bzero}\psi\right)(\bx) + \sum_{i=1}^{n_c}u_i\left(\mathcal{B}_{\mathbf{e}_i}\psi\right)(\bx).
	\end{equation}
	A finite-dimensional approximation of Eq.~\eqref{eqn:BilinearGeneratorModel_input} expresses the dynamics of each element in the dictionary $\boldsymbol{\psi}$ as a linear combination of $\boldsymbol{\psi}$'s elements,
	\begin{equation}\label{eqn:BilinearGeneratorModel_finite}
		\boldsymbol{\dot{\psi}}_{\mathbf{u}}(\bx) =  \mathbf{K}_{\mathbf{0}}\boldsymbol{\psi}(\bx) + \sum_{i=1}^{n_c}u_i\mathbf{B}_{\mathbf{e}_i}\boldsymbol{\psi}(\bx) + \mathbf{e}(\mathbf{x},\mathbf{u}),
	\end{equation}
	with some error $\mathbf{e}(\mathbf{x},\mathbf{u})$.
	
	In order to find the matrices $\mathbf{K}_{\mathbf{0}}$, $\mathbf{B}_{\mathbf{e}_1}$, $\ldots$, $\mathbf{B}_{\mathbf{e}_{n_c}}$, we shall minimize the errors $\mathbf{e}(\mathbf{x},\mathbf{u})$ in this finite-dimensional approximation over a training data set consisting of state-input-derivative tuples
	\begin{equation*}
	\left\lbrace \left(\mathbf{x}_j,\ \mathbf{u}_j,\ \mathbf{\dot{x}}_j = \mathbf{H}(\mathbf{x}_j, \mathbf{u}_j) \right) \right\rbrace_{j=1}^m.
	\end{equation*}
	Letting $\mathbf{B} = \begin{bmatrix} \mathbf{B}_{\mathbf{e}_1} & \cdots
          & \mathbf{B}_{\mathbf{e}_{n_c}} \end{bmatrix}$, we may express the error in Eq.~\eqref{eqn:BilinearGeneratorModel_finite} compactly as
	\begin{equation*}
	\mathbf{e}(\mathbf{x},\mathbf{u}) = \boldsymbol{\dot{\psi}}_{\mathbf{u}}(\bx) - 
	\begin{bmatrix} \mathbf{K}_{\mathbf{0}} & \mathbf{B} \end{bmatrix}
	\begin{bmatrix} \boldsymbol{\psi}(\bx)\\ \mathbf{u}\otimes\boldsymbol{\psi}(\bx) \end{bmatrix},
	\end{equation*}
	where $\otimes$ is the Kronecker product.
	It follows that minimizing the empirical error
	\begin{equation}\label{eqn:EDMD_generatorControlError}
	E(\mathbf{K}_{\mathbf{0}}, \mathbf{B}_{\mathbf{e}_1}, \ldots, \mathbf{B}_{\mathbf{e}_{n_c}}) = \sum_{j=1}^m \Vert \mathbf{e}(\mathbf{x}_j,\mathbf{u}_j) \Vert_2^2
	\end{equation}
	is a least-squares problem whose solution can be found using the training data.
	
	The time derivatives can be found according to the chain rule by differentiating the dictionary elements
	\begin{equation}
	\boldsymbol{\dot{\psi}}_j \triangleq \boldsymbol{\dot{\psi}}_{\mathbf{u}_j}(\bx_j) = \mathcal{D}_{\mathbf{x}}\boldsymbol{\psi}(\mathbf{x}_j)\mathbf{\dot{x}}_j,
	\label{eqn:ObservableDeriv_ChainRuleApprox}
	\end{equation}
	or by smoothed finite-differences using neighboring points along trajectories of the form
	\begin{equation}
	\boldsymbol{\dot{\psi}}_j = \frac{1}{\Delta t} \sum_{k=-s}^s c_k \boldsymbol{\psi}\left(\boldsymbol{\Phi}_{\mathbf{u}_j}^{k\Delta t}(\mathbf{x}_j)\right).
	\label{eqn:ObservableDeriv_FiniteDifferenceApprox}
	\end{equation}
	We suggest using Eq.~\eqref{eqn:ObservableDeriv_ChainRuleApprox} only when the exact state derivatives can be evaluated using the governing equations.
	Any noise in the state derivatives $\mathbf{\dot{x}}_j$ might be further amplified by the Jacobian matrix $\mathcal{D}_{\mathbf{x}}\boldsymbol{\psi}(\mathbf{x}_j)$,
	whereas Eq.~\eqref{eqn:ObservableDeriv_FiniteDifferenceApprox} can work directly with measurements of the observables along (noisy) trajectories.
	Constructing the matrices
	\begin{equation*}
		\begin{aligned}
			\bPsi_{\mathbf{X},\mathbf{U}} &=
			\begin{bmatrix} 
			\boldsymbol{\psi}(\bx^{(1)}) & \cdots & \boldsymbol{\psi}(\bx^{(m)}) \\
			\mathbf{u}^{(1)}\otimes\boldsymbol{\psi}(\bx^{(1)}) & \cdots & \mathbf{u}^{(m)}\otimes\boldsymbol{\psi}(\bx^{(m)})
			\end{bmatrix}
			\quad \text{and} \\
			\bPsidot_{\bX,\mathbf{U}} &=
			\begin{bmatrix} 
			\boldsymbol{\dot{\psi}}^{(1)} & \boldsymbol{\dot{\psi}}^{(2)} & \cdots & \boldsymbol{\dot{\psi}}^{(m)}
			\end{bmatrix},
		\end{aligned}
	\end{equation*}
	we find that a solution minimizing Eq.~\eqref{eqn:EDMD_generatorControlError} is given by
	\begin{equation}\label{eq:ContinuousEDMD_Matrices}
		\begin{bmatrix} \mathbf{K}_{\mathbf{0}} & \mathbf{B}_{\mathbf{e}_1} & \cdots & \mathbf{B}_{\mathbf{e}_{n_c}} \end{bmatrix} = 
		\bPsidot_{\bX,\mathbf{U}} \left(\bPsi_{\mathbf{X},\mathbf{U}}\right)^+.
	\end{equation}
	Observe that the amount of data required to make the matrix $\bPsi_{\mathbf{X},\mathbf{U}}$ full row-rank grows linearly with the dimension of the input $n_c$.
	Thus, the amount of data required to approximate the generators is comparable to the amount of data needed to approximate the Koopman operators at different fixed input levels as in Sec.~\ref{subsec:Koopman_operator_and_EDMD}.
	The advantage of using Eq.~\eqref{eq:ContinuousEDMD_Matrices} is that the model can be learned from a data set where the inputs $\mathbf{u}_j$ take arbitrary values in $\mathcal{U}$.
	
	With the above EDMD-like approximation via Eq.~\eqref{eq:ContinuousEDMD_Matrices}, we can now construct a surrogate model from data. 
	To this end, we will keep track of the observables $\mathbf{z} = \boldsymbol{\psi}(\mathbf{x})$ and evolve them according to the bilinear control system
	\begin{equation}\label{eqn:Continuous_EDMD_ROM}
		\boxed{
			\mathbf{\dot{z}}(t) = \mathbf{K}_{\mathbf{u}(t)}\mathbf{z}(t) = \left( \mathbf{K}_{\mathbf{0}} + \sum_{i=1}^{n_c} u_i(t)\mathbf{B}_{\mathbf{e}_i} \right) \mathbf{z}(t).
		}
	\end{equation}
	If a small and informative dictionary of observables $\boldsymbol{\psi}$ can be found, then it will be much more efficient to model the dynamics of the vector of observations $\bz = \boldsymbol{\psi}(\bx)$ using \eqref{eqn:Continuous_EDMD_ROM} than to evolve the full state $\bx$ directly using the original dynamics.
	Moreover, the affine property allows us to learn and compute with the reduced model in a very efficient way.
	
	\subsection{Discrete vs.\ Continuous Approximation Methods}
	\label{subsec:PracticalApproximation}
	So far, we have exploited the input-affine property of the Koopman generators in order to construct the bilinear reduced order model given in Eq.~\eqref{eqn:Continuous_EDMD_ROM}.
	In this section we will discuss how discrete approximations of the these dynamics are connected with models based on finite-time Koopman operators.
	To make this connection, we shall place a zero-order hold on the input over intervals of length $\Delta t$.
	In this case, the continuous-time dynamics in the space of observables become linear over each interval, where it can be understood using Theorem \ref{thm:KoopmanGenerator} in terms of the finite-time Koopman operator.
	
	Working with models based on approximating the finite-time Koopman operator appears advantageous for two main reasons. 
	First, approximating the finite-time Koopman operator on a finite dictionary does not require differentiating noisy signals. 
	Second, the models obtained are discrete-time, requiring only a single matrix-vector product to evolve the chosen observables over an interval.
	This is in contrast to the continuous-time models obtained using generator-based approaches, which must be integrated over each interval. 
	
	On these grounds, it is worth investigating when and how approximations of the Koopman generator can be replaced by approximations of the finite-time Koopman operator for the purpose of control-oriented reduced-order modeling.
	In short, this is possible when the time interval $\Delta t$ of the zero-order hold is small or when the dynamics and observables are affine.
	Furthermore, we will show that EDMD-like modeling based on approximating finite-time Koopman operators is actually equivalent to a special case of the generator-based approach described in Section \ref{subsec:EDMDforGenerators}.
	
	Example \ref{eg:KoopmanCounterExample} suggests that while the finite-time Koopman operators are not affine with respect to the input, such a property might be true to first order in $\Delta t$.
	Indeed, examining the series expansion in Theorem \ref{thm:KoopmanGenerator} for $\mathcal{K}_{\mathbf{u}}^{\Delta t}$, one obtains
	\begin{equation}
         \label{eq:1}
         \begin{aligned}
			\mathcal{K}_{\mathbf{u}}^{\Delta t} &= \exp\left(\mathcal{K}_{\mathbf{0}}\Delta t + \sum_{i=1}^{n_c}u_i(\mathcal{K}_{\mathbf{e}_i} - \mathcal{K}_{\mathbf{0}})\Delta t\right)\\
			&= \mathcal{I} + \mathcal{K}_{\mathbf{0}}\Delta t + \sum_{i=1}^{n_c}u_i\left[(\mathcal{I} + \mathcal{K}_{\mathbf{e}_i}\Delta t) - (\mathcal{I} + \mathcal{K}_{\mathbf{0}}\Delta t)\right] + \mathcal{O}(\Delta t^2)\\
			&= e^{\mathcal{K}_{\mathbf{0}}\Delta t} + \sum_{i=1}^{n_c}u_i\left(e^{\mathcal{K}_{\mathbf{e}_i}\Delta t} - e^{\mathcal{K}_{\mathbf{0}}\Delta t}\right) + \mathcal{O}(\Delta t^2)\\
			&= \mathcal{K}_{\mathbf{0}}^{\Delta t} + \sum_{i=1}^{n_c}u_i\left(\mathcal{K}_{\mathbf{e}_i}^{\Delta t} - \mathcal{K}_{\mathbf{0}}^{\Delta t}\right) + \mathcal{O}(\Delta t^2).
		\end{aligned}
	\end{equation}
	This means that the EDMD-like method and continuous time model presented in Sec.~\ref{subsec:EDMDforGenerators} still retain their accuracy to first order in $\Delta t$ when approximations of the Koopman generator are replaced by the finite-time Koopman operator.
	
	To make this claim precise, suppose that we know the future states $\mathbf{\tilde{x}}_j = \boldsymbol{\Phi}_{\mathbf{u}_j}^{\Delta t}(\mathbf{x}_j)$ for each state-input pair $\mathbf{x}_j$, $\mathbf{u}_j$ in our training data set and construct the matrix
	\begin{equation*}
		\begin{aligned}
			\boldsymbol{\Psi}_{\mathbf{\tilde{X}},\mathbf{U}} &=
			\begin{bmatrix} 
			\boldsymbol{\psi}(\mathbf{\tilde{x}}_1) & \boldsymbol{\psi}(\mathbf{\tilde{x}}_2) & \cdots & \boldsymbol{\psi}(\mathbf{\tilde{x}}_m)
			\end{bmatrix}.
		\end{aligned}
	\end{equation*}
	Using the same argument as in Sec.~\ref{subsec:EDMDforGenerators}, we obtain EDMD-like approximations $\mathbf{K}_{\mathbf{0}}^{\Delta t}$, $\mathbf{K}_{\mathbf{e}_1}^{\Delta t}$, $\ldots$, $\mathbf{K}_{\mathbf{e}_{n_c}}^{\Delta t}$ for the finite-time Koopman operators at input levels $\mathbf{0}$, $\mathbf{e}_1$, $\ldots$, $\mathbf{e}_{n_c}$, by defining $\mathbf{B}_{\mathbf{e}_i}^{\Delta t} = \mathbf{K}_{\mathbf{e}_i}^{\Delta t} - \mathbf{K}_{\mathbf{0}}^{\Delta t}$ and computing
	\begin{equation}\label{DiscreteEDMD_Matrices}
		\begin{bmatrix} \mathbf{K}_{\mathbf{0}}^{\Delta t} & \mathbf{B}_{\mathbf{e}_1}^{\Delta t} & \cdots & \mathbf{B}_{\mathbf{e}_{n_c}}^{\Delta t} \end{bmatrix} = 
		\boldsymbol{\Psi}_{\mathbf{\tilde{X}},\mathbf{U}} \left(\bPsi_{\mathbf{X},\mathbf{U}}\right)^+.
	\end{equation}
	The resulting discrete-time model for the evolution of observables $\mathbf{z}_k = \boldsymbol{\psi}(\mathbf{x}(k\Delta t))$ under the dynamics with constant input $\mathbf{u}$ over $\Delta t$ is given by
	\begin{equation}
		\label{eqn:Discrete_EDMD_ROM}
		\boxed{
			\mathbf{z}_{k+1} = \left( \mathbf{K}_{\mathbf{0}}^{\Delta t} + \sum_{i=1}^{n_c} u_i\mathbf{B}_{\mathbf{e}_i}^{\Delta t} \right)\mathbf{z}_k +\mathcal{O}(\Delta t^2).
		}
	\end{equation}
	Note that the input-affine approximation for the finite-time Koopman operators in the above equation cannot achieve more than first-order accuracy in the general case.
	\begin{remark}
	If greater accuracy is desired, more terms in the series expansion
        employed in~(\ref{eq:1}) may be retained, to obtain an expression
        analogous to~(\ref{eqn:Discrete_EDMD_ROM}), but nonlinear in the
        components $u_i$.
	\end{remark}
	
	The following Theorem \ref{thm:AffineKoopmanAndEulerIntegration} shows that the above technique based on approximating finite-time Koopman operators is (under a mild assumption) equivalent to a particular discretization of the method in Sec.~\ref{subsec:EDMDforGenerators}.
	In particular, it is equivalent to using a first-order forward difference in Eq.~\eqref{eqn:ObservableDeriv_FiniteDifferenceApprox} and explicit Euler integration to advance the model in Eq.~\eqref{eqn:Continuous_EDMD_ROM}.
	\begin{theorem}[Discrete and Continuous Models]
		\label{thm:AffineKoopmanAndEulerIntegration}
		Assume that $\bPsi_{\mathbf{X},\mathbf{U}}$ has full row rank and that approximations of the finite-time Koopman operators and infinitessimal generators are computed using Eq.~\eqref{DiscreteEDMD_Matrices} and Eq.~\eqref{eq:ContinuousEDMD_Matrices} with the finite-difference formula
		\begin{equation}
                  \label{eq:2}
                  \boldsymbol{\dot{\psi}}_j = \frac{\boldsymbol{\psi}(\mathbf{\tilde{x}}_j) - \boldsymbol{\psi}(\mathbf{x}_j)}{\Delta t}.
		\end{equation}
		Then advancing the dynamics via the discrete-time model Eq.~\eqref{eqn:Discrete_EDMD_ROM} is identical to advancing the dynamics in the continuous-time model Eq.~\eqref{eqn:Continuous_EDMD_ROM} using explicit Euler integration over $\Delta t$.
		Furthermore, the matrix approximations obtained using the two methods are related by the identities
		\begin{equation}
		\mathbf{K}_{\mathbf{0}}^{\Delta t} = \mathbf{I}_{n_o} + \Delta t \mathbf{K}_{\mathbf{0}}, 
		\quad \mbox{and}\quad
		\mathbf{B}_{\mathbf{e}_i}^{\Delta t} = \Delta t \mathbf{B}_{\mathbf{e}_i},
		\label{eqn:DiscreteContinuousEDMD_compatibility}
		\end{equation}
		for all $i=1$, $\ldots$, $n_c$.
		Hence, they are compatible in the sense of the operator-generator relationship $\mathbf{K}_{\mathbf{u}}^{\Delta t} = \exp(\mathbf{K}_{\mathbf{u}}\Delta t)$ to first order in $\Delta t$.
	\end{theorem}
	\begin{proof}
		It suffices to show that the identities in Eq.~\eqref{eqn:DiscreteContinuousEDMD_compatibility} hold when the matrix approximations are computed using the two EDMD-like approximation methods described in Eq.~\eqref{eq:ContinuousEDMD_Matrices} and Eq.~\eqref{DiscreteEDMD_Matrices}.
		If these identities hold then it is clear that Eq.~\eqref{eqn:Discrete_EDMD_ROM} is identical to an explicit Euler scheme
		\begin{equation*}
		\mathbf{z}(t+\Delta t) = \mathbf{z}(t) + 
		\Delta t \left( \mathbf{K}_{\mathbf{0}} + \sum_{i=1}^{n_c} u_i\mathbf{B}_{\mathbf{e}_i} \right) \mathbf{z}(t)
		\end{equation*}
		for integrating Eq.~\eqref{eqn:Continuous_EDMD_ROM}.
		Compatibility is also clear from the identities together with series expansions of the matrix exponentials.
		Using the finite-difference formula~(\ref{eq:2}), we observe that
		\begin{equation*}
		\boldsymbol{\Psi}_{\mathbf{\tilde{X}},\mathbf{U}} =
		\begin{bmatrix}
		\mathbf{I}_{n_o} & \mathbf{0}_{n_o\times n_c n_o} 
		\end{bmatrix}
		\boldsymbol{\Psi}_{\mathbf{X},\mathbf{U}} 
		+ \Delta t\boldsymbol{\dot{\Psi}}_{\mathbf{X},\mathbf{U}}.
		\end{equation*}
		Since $\boldsymbol{\Psi}_{\mathbf{X},\mathbf{U}}$ was assumed to have full row rank, $(\boldsymbol{\Psi}_{\mathbf{X},\mathbf{U}})^+$ is a right inverse. 
		Multiplying on the right by $(\boldsymbol{\Psi}_{\mathbf{X},\mathbf{U}})^+$, we obtain
		\begin{equation*}
		\boldsymbol{\Psi}_{\mathbf{\tilde{X}},\mathbf{U}}(\boldsymbol{\Psi}_{\mathbf{X},\mathbf{U}})^+ = 
		\begin{bmatrix}
		\mathbf{I}_{n_o} & \mathbf{0}_{n_o\times n_c n_o} 
		\end{bmatrix}
		+ \Delta t\boldsymbol{\dot{\Psi}}_{\mathbf{X},\mathbf{U}}(\boldsymbol{\Psi}_{\mathbf{X},\mathbf{U}})^+.
		\end{equation*}
		Substituting Eq.~\eqref{eq:ContinuousEDMD_Matrices} and Eq.~\eqref{DiscreteEDMD_Matrices} into the above equation, we obtain the desired identities.
	\end{proof}
	
	As we have seen, for general control-affine systems, bilinear models based on approximate finite-time Koopman operators cannot be justifiably used unless the time step is sufficiently small.
	However, for linear systems, the finite-time Koopman operators are control-affine over the invariant subspace of affine observables.
	In other words, when the underlying dynamics are linear, then we can justifiably make control-affine approximations of the finite-time Koopman operators when the observables $\boldsymbol{\psi}$ are affine.
	This fact may seem trivial, but modeling the system this way could achieve significant dimensionality reduction if an approximately invariant subspace of affine observables can be found.
	
	To illustrate this idea, consider the linear time-invariant system
	\begin{equation*}
	\mathbf{\dot{x}} = \mathbf{A}\mathbf{x} + \mathbf{B}\mathbf{u},
	\end{equation*}
	whose constant-input flow map is given by the well-known formula
	\begin{equation*}
	\boldsymbol{\Phi}_{\mathbf{\bar{u}}}^{\Delta t}(\mathbf{x}_0) = e^{\mathbf{A} \Delta t}\mathbf{x}_0 + \left(\int_{0}^{\Delta t}e^{\mathbf{A} (\Delta t - \tau)}\mathbf{B}\ d\tau \right) \mathbf{\bar{u}}.
	\end{equation*}
	Clearly, the flow map is affine with respect to the input, since
	\begin{align*}
		\boldsymbol{\Phi}_{\alpha_1\mathbf{\bar{u}}_1 + \alpha_2\mathbf{\bar{u}}_2}^{\Delta t} = \boldsymbol{\Phi}_{\mathbf{0}}^{\Delta t} + 
		\alpha_1\left[\boldsymbol{\Phi}_{\mathbf{\bar{u}}_1}^{\Delta t} - \boldsymbol{\Phi}_{\mathbf{0}}^{\Delta t}\right] + \alpha_2\left[\boldsymbol{\Phi}_{\mathbf{\bar{u}}_2}^{\Delta t} - \boldsymbol{\Phi}_{\mathbf{0}}^{\Delta t}\right],
	\end{align*}
	and with respect to the state, since
	\begin{align*}
		\boldsymbol{\Phi}_{\mathbf{\bar{u}}}^{\Delta t}(\alpha_1 \mathbf{x}_1 + \alpha_2 \mathbf{x}_2) = 
		\boldsymbol{\Phi}_{\mathbf{\bar{u}}}^{\Delta t}(\mathbf{0}) +
		\alpha_1 \left[\boldsymbol{\Phi}_{\mathbf{\bar{u}}}^{\Delta t}(\mathbf{x}_1) - \boldsymbol{\Phi}_{\mathbf{\bar{u}}}^{\Delta t}(\mathbf{0}) \right] + \alpha_2 \left[\boldsymbol{\Phi}_{\mathbf{\bar{u}}}^{\Delta t}(\mathbf{x}_2) - \boldsymbol{\Phi}_{\mathbf{\bar{u}}}^{\Delta t}(\mathbf{0}) \right]
	\end{align*}
	over any length of time $\Delta t$.
	It follows immediately that the space of affine observables $\mathcal{V} = \lbrace \psi:\mathbf{x}\mapsto c + \mathbf{w}^T\mathbf{x}\ :\ c\in\mathbb{R},\ \mathbf{w}\in\mathbb{R}^n \rbrace$ is invariant under composition with any constant-input flow map, i.e. $\psi \circ \boldsymbol{\Phi}_{\mathbf{\bar{u}}}^{\Delta t} \in\mathcal{V}$ for every $\psi\in\mathcal{V}$.
	Therefore, the entire family of finite-time Koopman operators $\mathcal{K}_{\mathbf{\bar{u}}}^{\Delta t}$ over constant inputs $\mathbf{\bar{u}}$ and intervals $\Delta t$ can be defined on $\mathcal{V}$.
	Over this invariant space of affine observables the finite-time Koopman operators are also input-affine, yielding the identity
	\begin{equation*}
	\mathcal{K}_{\mathbf{u}}^{\Delta t} = \mathcal{K}_{\mathbf{0}}^{\Delta t} + \sum_{i=1}^{n_c}u_i \left(\mathcal{K}_{\mathbf{e}_i}^{\Delta t} - \mathcal{K}_{\mathbf{0}}^{\Delta t} \right),
	\end{equation*}
	which holds for any $\Delta t$.
	Note also that if the finite-time Koopman operators are input affine and have $\mathcal{V}$ as an invariant subspace, then the original system must be linear.
	
	\begin{remark}
		In practice, the input-affine interpolation approach for the continuous and finite-time models show satisfactory results even for systems where the control input enters in a weakly nonlinear manner. 
		The loss in accuracy obviously depends on the degree of nonlinearity, but on short time horizons, the MPC often yields the desirable control input.
	\end{remark}
	
	\subsection{Example: forced Duffing Equation}
	\label{subsec:GenDuffing}
	
	As a first example, we consider the Duffing equation
	\begin{equation}\label{eq:Duffing}
	\bxdot(t) = \left(\begin{array}{c}
	x_2(t) \\ -\delta x_2(t) - \alpha x_1(t) - \beta x_1^3(t) + u(t)
	\end{array}\right),
	\end{equation}
	with an affine control input $u(t) \in \Ucal = [-1,1]$. 
	For the observable $\boldsymbol{\psi}$ we use all monomials in the state variables $x_1$ and $x_2$ up to degree five, which results in a 21-dimensional linear system in feature space. 
	For the data collection, we randomly select 100 initial conditions from the interval $[-3,3]^2\subset \R^2$ and evaluate the right-hand side at these points for both $u(t) = \bar{u}_1 = -1$ and $u(t)=\bar{u}_2 = 1$.
	We then test both formulations for constructing the K-ROM. In the first approach, we restrict the input (and consequently, the data collection) to $\hat{\Ucal} = \{-1,1\}$ and compute the matrix $\bB_{1}$ via the difference between the two corresponding Koopman operators, i.e., $\bB_{1} = \frac{1}{2}(\bK_{1}-\bK_{-1})$. 
	In the second approach, we use a continuous input signal according to Eq.~\eqref{eq:ContinuousEDMD_Matrices}. 
	
	We observe that regardless of the approximation approach, the system can be predicted very accurately for approximately one second for most initial conditions $\bxi_0$, one of which is visualized in Fig.~\ref{fig:DuffingPrediction}. 
	There, the prediction of the states $x_1$ and $x_2$ is compared in (a) and (b), and the error for the $x_1$ component is shown in (c).
	We compare the solutions for the inputs $u(t) = \pm 1$ at which the data was collected (i.e., for the individual generator approximations) as well as for interpolated values $u(t) = 0$ and $u(t) = \sin(\pi t)$. 
	Hence, it can be concluded that the incorporation into an MPC scheme is promising. The specific choice of the modeling approach should depend on the problem setup.
	
	\begin{figure}
		\centering
		\parbox[b]{0.48\textwidth}{\centering (a) \\ \includegraphics[width=.43\textwidth]{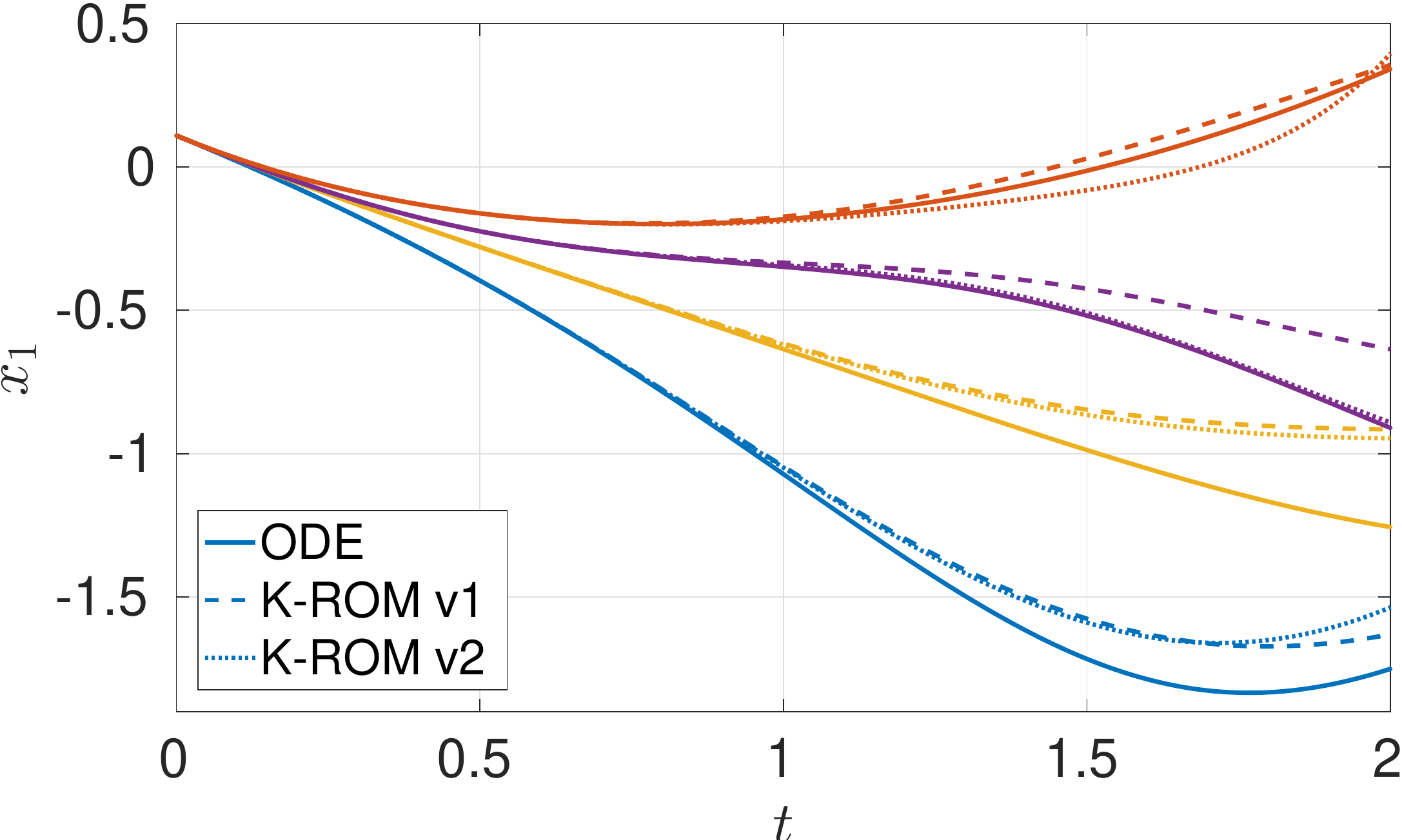}}
		\parbox[b]{0.48\textwidth}{\centering (b) \\ \includegraphics[width=.46\textwidth]{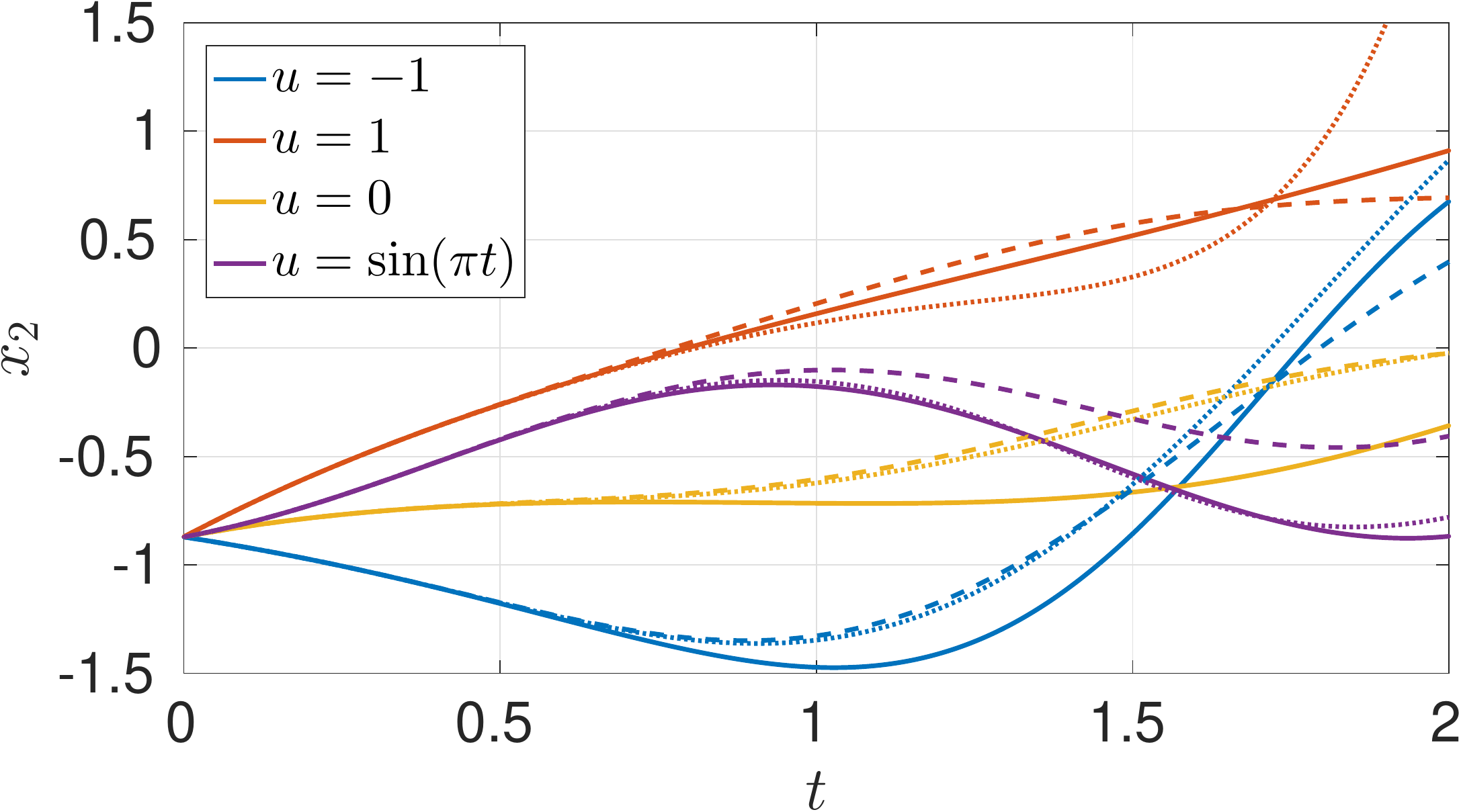}}
		\parbox[b]{0.48\textwidth}{\centering (c) \\ \includegraphics[width=.45\textwidth]{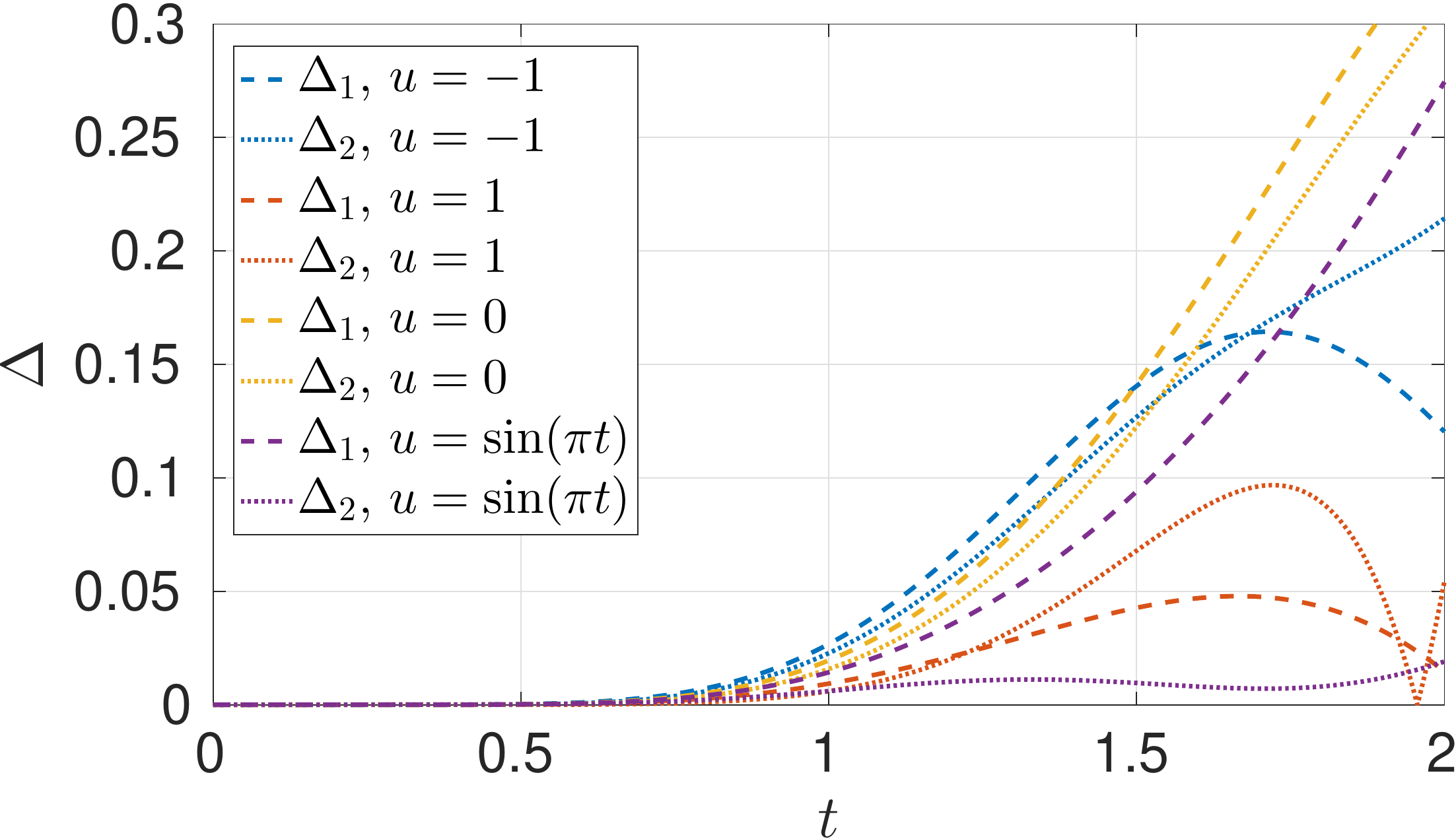}}
		\caption{Comparison between the true solution and the surrogate model for two different K-ROM approximations, i.e., construction via two Koopman operators at $u\pm1$ (v1, dashed lines) or construction according to Eq.~\eqref{eq:ContinuousEDMD_Matrices} (v2, dotted lines). In both cases, a dictionary of monomials up to order five was used. (a) $x_1$. (b) $x_2$. (c) Error $\Delta$ for $x_1$.}
		\label{fig:DuffingPrediction}
	\end{figure}
	
	\subsection{Example: Cylinder flow}
	\label{subsec:PDE}
	The approach shows its full potential when considering PDEs, which we will illustrate in the following example of a cylinder flow, governed by the two-dimensional incompressible Navier--Stokes equations at a moderate Reynolds number of $Re = 100$ (see Fig.~\ref{fig:vonKarman}~(a) for the problem setup): 
	\begin{align*}
		\dot{\bv}(\bxi,t) + \bv(\bxi,t) \cdot \nabla \bv(\bxi,t) &= \nabla p(\bxi,t) + \frac{1}{Re} \Delta \bv(\bxi,t), \\
		\nabla \cdot \bv(\bxi,t) &= \bzero, \\
		\by(\bxi,t_0) &= \by^0(\bxi),
	\end{align*}
	where $\bv$ and $p$ are the flow velocity and the pressure, respectively, depending on space $\bxi$ and time $t$. 
	The problem is discretized using a finite volume method and the PISO scheme for time integration \cite{FP02}. 
	All calculations are performed using the open source solver OpenFOAM \cite{JJT07}, and as we do not have explicit access to the time derivative, we use the first-order approximation via the Koopman operator as introduced in Sec.~\ref{subsec:PracticalApproximation}; i.e., predictions are performed using Eq.~\eqref{eqn:Discrete_EDMD_ROM}.
	
	The system is controlled via rotation of the cylinder; i.e., $u(t)$ is the angular velocity.
	The uncontrolled system possesses a periodic solution, the well-known \emph{von K\'{a}rm\'{a}n vortex street}.
	
	In order to obtain a low-dimensional reduced model and achieve a high acceleration, we construct the K-ROM only for the control relevant quantities, i.e., the lift  $L$ and drag force $D$ of the cylinder acting in vertical and horizontal direction, respectively. 
	Additionally, we observe the vertical velocity at six different positions $(\bxi_1,\ldots,\bxi_6)$ in the cylinder wake (see Fig.~\ref{fig:vonKarman}~(a)):
	\begin{equation*}
		\begin{aligned}
		\boldsymbol{\hat{\psi}}((\bv(\cdot,t), p(\cdot,t)) = \left(L(t), D(t), v_2(\bxi_1,t), \ldots, v_2(\bxi_6,t)\right).
		\end{aligned}
	\end{equation*}
	For a more accurate approximation via EDMD, we use as the observable $\boldsymbol{\psi}$ all monomials of $\boldsymbol{\hat{\psi}}$ up to order 2, which results in a K-ROM of dimension 45. 
	We collect data from one long term simulation with random control inputs $u_i\in\hat{\Ucal}=\{0,2\}$. 
	The time step in the time series is $0.25$ in comparison to the time step $0.01$ in the finite volume scheme. 
	Figs.~\ref{fig:vonKarman}~(b) and \ref{fig:vonKarman}~(c) show a comparison between the PDE and the K-ROM solution for constant control inputs $\ubar_0 = 0$ and $\ubar_1 = 2$. 
	We see that in both cases, the solutions agree remarkably well, considering that the finite volume discretization for this case consists of $22,000$ cells (i.e., $66,000$ unknowns per time step) and the K-ROM is a 45-dimensional linear model. 
	This results in a speed-up of approximately five orders of magnitude (OpenFOAM vs.~MATLAB). 
	Fig.~\ref{fig:vonKarman}~(d) shows a comparison of the solutions for a
        sinusoidal control and we see that the agreement is still satisfactory
        even though the system is not control affine and the K-ROM is only
        first-order accurate.
        
        \begin{figure}
        	\centering
        	\parbox[b]{0.48\textwidth}{\centering (a) \\ \includegraphics[width=.40\textwidth]{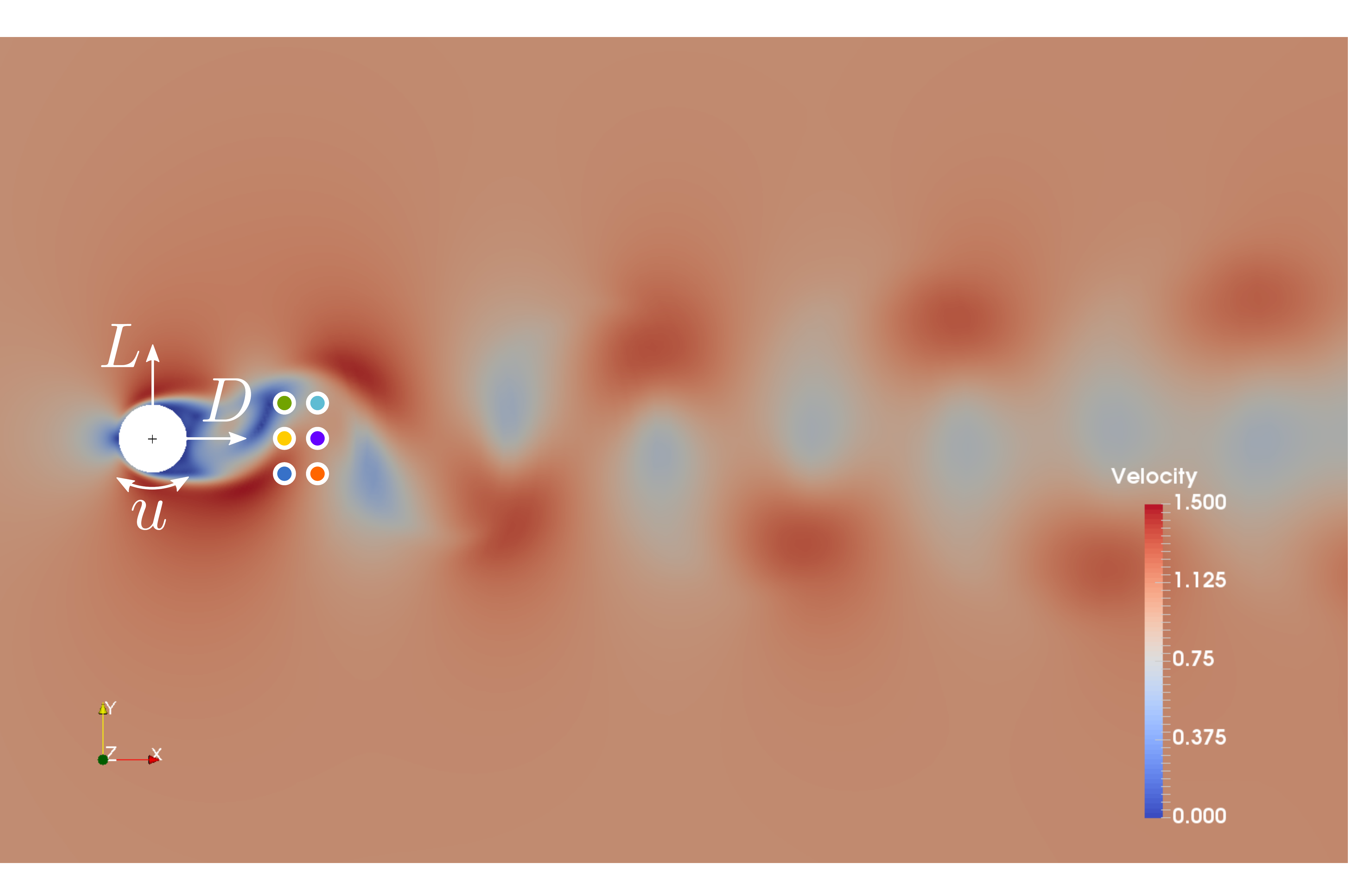}}
        	\parbox[b]{0.48\textwidth}{\centering (b) \\ \includegraphics[width=.40\textwidth]{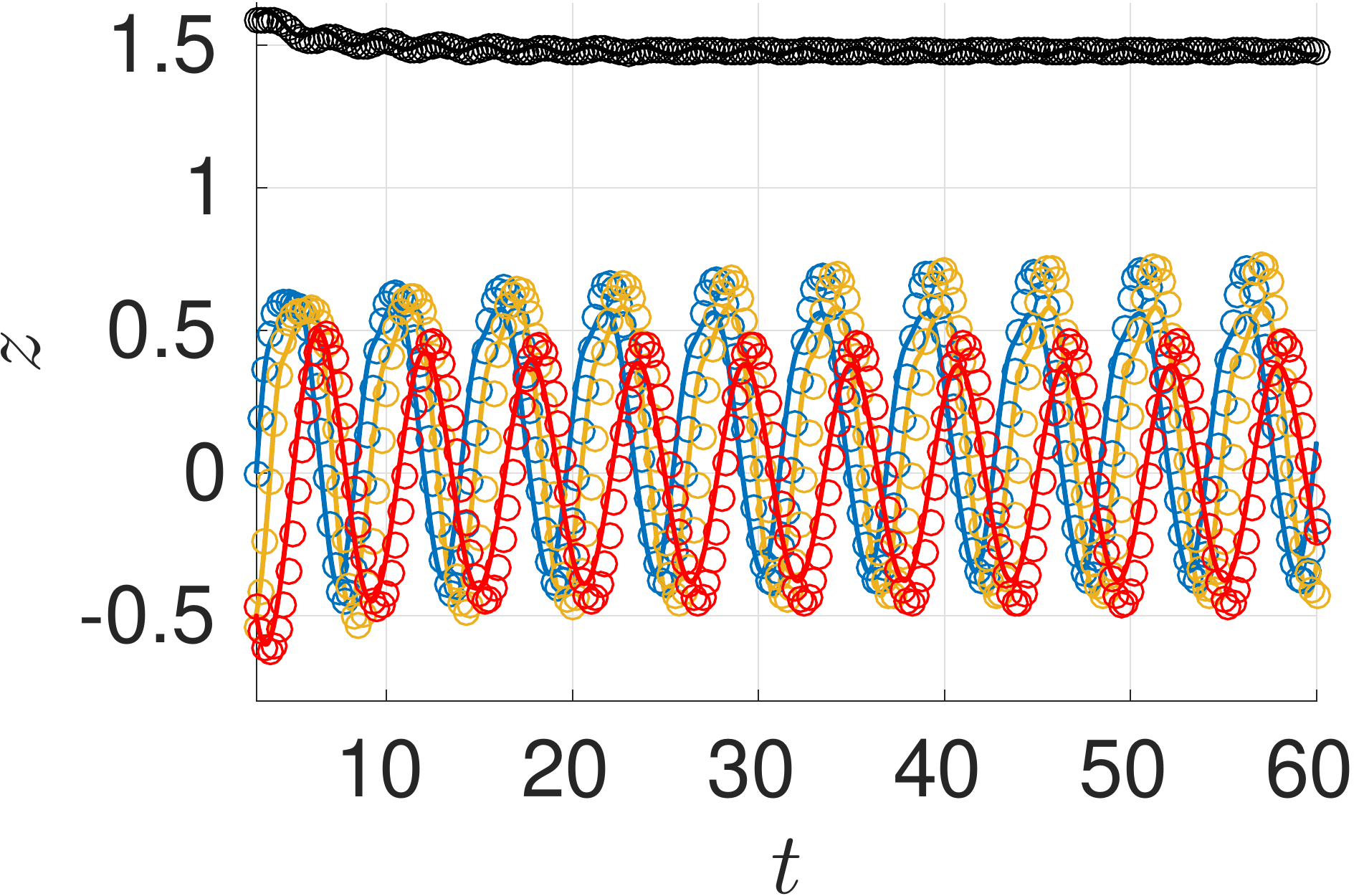}} \\[2ex]
        	\parbox[b]{0.48\textwidth}{\centering (c) \\ \includegraphics[width=.40\textwidth]{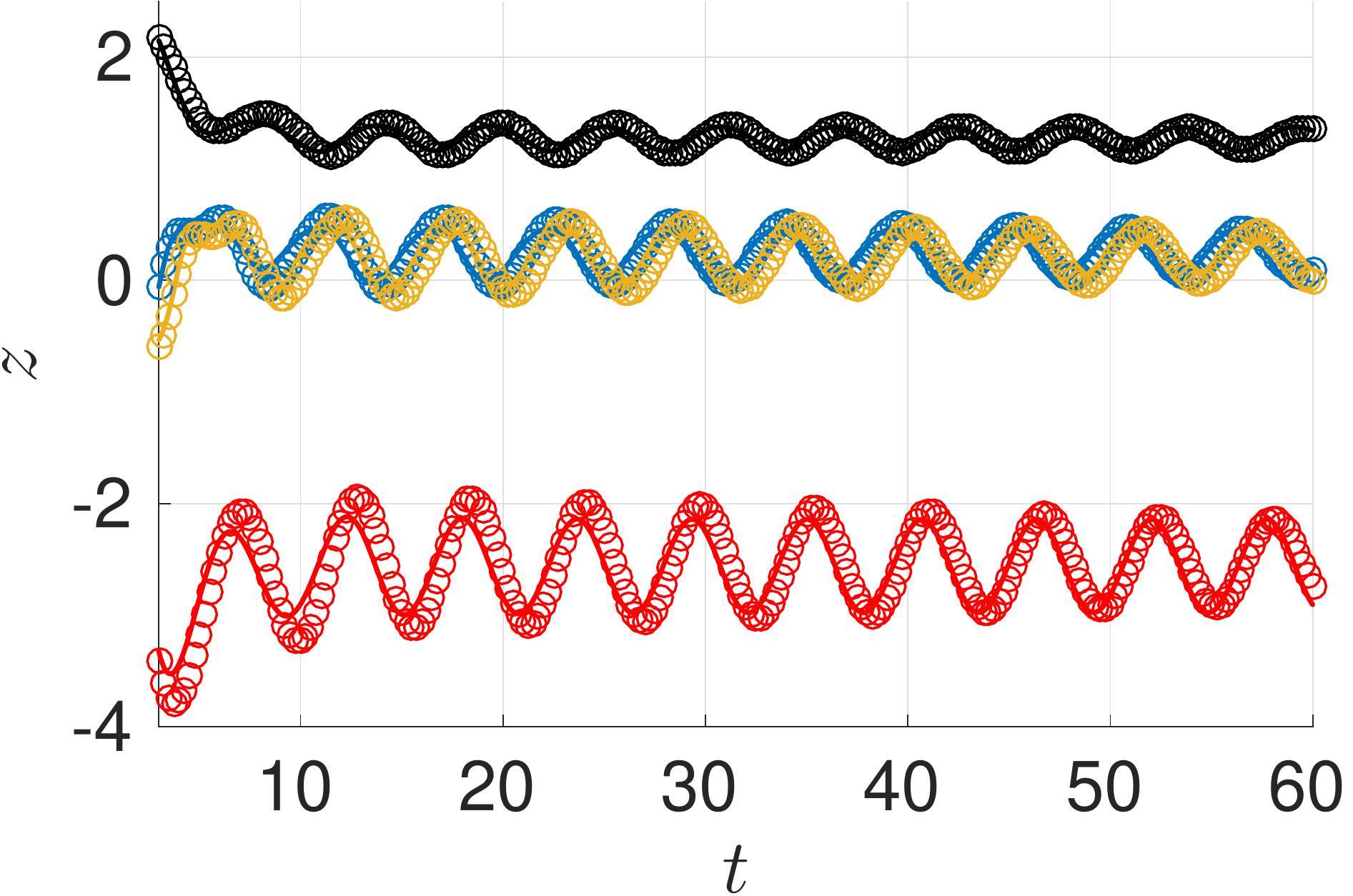}}
        	\parbox[b]{0.48\textwidth}{\centering (d) \\ \includegraphics[width=.40\textwidth]{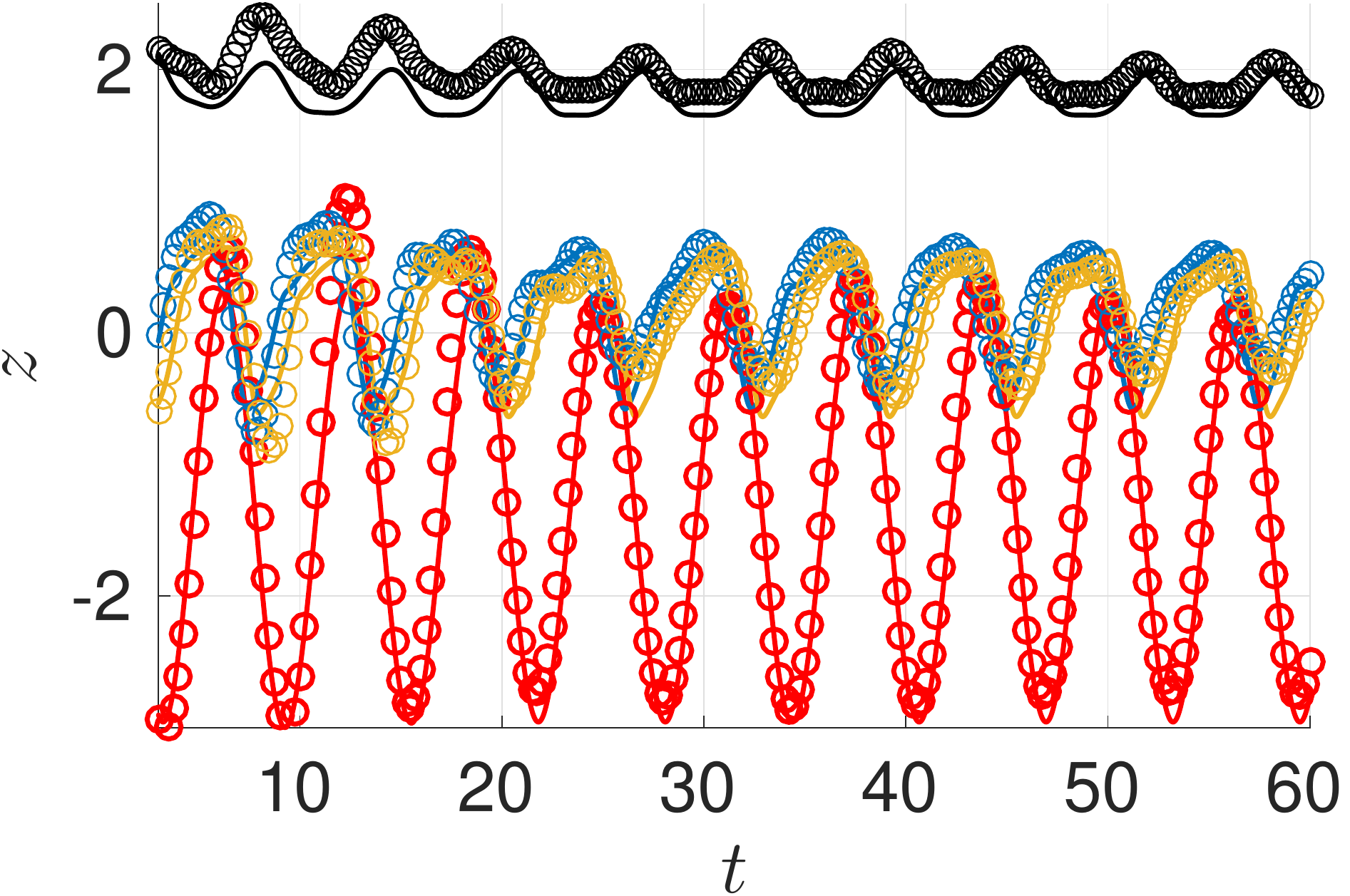}}
        	\caption{(a) Sketch of the problem setting. The system is controlled by rotating the cylinder. (b) and (c) Observation of the PDE solution and K-ROM approximation for $\ubar_a = 0$ and $\ubar_b = 2$, respectively. The lift is shown by the red line, the drag is shown in black and the two remaining lines are the vertical velocities at two of the six positions shown in (a). (d) Observation of the PDE solution and K-ROM approximation for $u(t) = 1 + \sin(t)$.
        	}
        	\label{fig:vonKarman}
        \end{figure}
	
	\section{MPC solution using EDMD-based ROMs}
	\label{sec:MPCsolution}
	
	Having shown the predictive capabilities of the affine Koopman generator approach \eqref{eqn:Continuous_EDMD_ROM} as well as the first-order approximation via the Koopman operator \eqref{eqn:Discrete_EDMD_ROM}, we can use these models in an MPC framework for feedback control.
	Similar to the K-ROM approximation \eqref{eq:MPC_STO_Koopman} of the
        mixed-integer problem \eqref{eq:MPC_STO}, we can introduce a K-ROM
        approximation of the original problem~\eqref{eq:MPC}.  To this end, we
        restrict the allowable input signals $\bu$ to
        lie in a finite-dimensional
        subspace of $L^2([t_0,t_e],\mathcal{U})$.  For instance, if we
        discretize the signal $\mathbf{u}(t)$
        using zero-order holds over intervals of length $\Delta t$, then we may
        choose this subspace to be all piecewise constant functions over
        intervals $I_k = t_0 + [k, k+1)\Delta t$, which we denote
	\begin{equation*}
	\mathcal{F}([t_0, t_e], \mathcal{U}) \triangleq \left\lbrace \sum_{k=0}^{\lceil (t_e-t_0)/\Delta t \rceil}\bu_k \mathbf{1}_{I_k}(t):\ \bu_k\in\mathcal{U} \right\rbrace.
	\end{equation*}
	Alternatively, our actuators might have limited bandwidth, in which case
        it might make sense to discretize the input as a Fourier series up to
        some maximum frequency $\omega_l$, and this would give rise to a
        different finite-dimensional subspace.
        We then approximate the original optimal control problem~(\ref{eq:MPC}) as
	\begin{equation} \label{eq:MPC_Koopman} \tag{K-MPC}
	\begin{aligned}
	\min_{\bu \in \mathcal{F}([t_0, t_e], \mathcal{U})} & J = \int_{t_0}^{t_e} \hat{L}(\bz(t), \bu(t), t)\ dt \\
	\mbox{s.t.}\quad \dot{\bz}(t) &= \left(\bK_{\bzero}+ \sum_{i=1}^{n_c} u_i(t) \bB_{\mathbf{e}_i}\right) \bz(t) \\
	\bz(0) &= \boldsymbol{\psi}(\bx_0).
	\end{aligned}
	\end{equation}
	
	In this section, we will focus on solving the general problem~\eqref{eq:MPC_Koopman}.
	An adjoint-based approach will be used to efficiently obtain gradients of the objective with respect to the input signal subject to the dynamical constraints imposed by the K-ROM.
	These gradients together with the linearized reduced-order model and adjoint equations yield necessary first-order conditions for optimality.
	We will also see that the bilinear structure of the K-ROM enables an efficient Newton-type solver for these first-order optimality conditions.
	Alternatively, if there are additional constraints on the input or allowable states of the system, sequential quadratic programming or interior-point methods will also benefit from the derived gradient.
	
	\subsection{Gradient and first-order optimality}
	In order to perform optimization, let us expand the input signal
	\begin{equation*}
	\mathbf{u}(t) = \sum_{k=1}^d \hat{u}_k \boldsymbol{\varphi}_k(t)
	= \boldsymbol{\varphi}(t)\mathbf{\hat{u}}
	\end{equation*}
	in an orthonormal basis $\lbrace \boldsymbol{\varphi}_1$, $\ldots$, $\boldsymbol{\varphi}_d\rbrace$ for $\mathcal{F}([t_0, t_e], \mathcal{U})$.
	Once the gradient $\boldsymbol{\nabla}_{\mathbf{u}}J$ has been computed with respect to a general $\mathbf{u}\in L^2([t_0, t_e], \mathcal{U})$, it is easy to obtain the gradient over the subspace $\mathcal{F}([t_0, t_e], \mathcal{U})$ by computing
	\begin{equation*}
	\nabla_{\hat{u}_k}J = \left\langle \boldsymbol{\varphi}_k,\ \boldsymbol{\nabla}_{\mathbf{u}}J \right\rangle.
	\end{equation*}
	
	In order to find the gradient of the objective $J$ with respect to the input signal $\mathbf{u}$, let us introduce small perturbations $\delta\mathbf{z}$ and $\delta \mathbf{u}$ that are related by the linearized model dynamics
	\begin{equation*}
		\left(\frac{d}{dt} - \mathbf{K}_{\mathbf{0}} - \sum_{i=1}^{n_c}u_i(t) \mathbf{B}_{\mathbf{e}_i} \right)\delta \mathbf{z}(t) = \begin{bmatrix}
		\mathbf{B}_{\mathbf{e}_1}\mathbf{z}(t) & \cdots & \mathbf{B}_{\mathbf{e}_{n_c}}\mathbf{z}(t)
		\end{bmatrix} \delta \mathbf{u}(t).
	\end{equation*}
	Since we are assuming the initial condition is fixed, the initial condition for the perturbation is zero $\delta \mathbf{z}(0) = \mathbf{0}$.
	
	Denoting the inner product of real vector-valued signals by
	\begin{equation*}
	\langle \mathbf{z},\ \boldsymbol{\lambda} \rangle = \int_{t_0}^{t_e} \mathbf{z}(t)^T \boldsymbol{\lambda}(t)\ dt,
	\end{equation*}
    the change in the objective corresponding to these small perturbations is given by
	\begin{align*}
	\delta J = \langle \boldsymbol{\gamma},\ \delta \mathbf{z} \rangle &+ \langle \boldsymbol{\rho},\ \delta \mathbf{u} \rangle, \\
	\mbox{where} \quad \boldsymbol{\gamma}(t) = \boldsymbol{\nabla}_{\mathbf{z}}\hat{L}(\mathbf{z}(t), \mathbf{u}(t), t) \quad &\mbox{and}\quad \boldsymbol{\rho}(t) = \boldsymbol{\nabla}_{\mathbf{u}}\hat{L}(\mathbf{z}(t), \mathbf{u}(t), t).
	\end{align*}
	We want to express the first inner product in the above equation in terms of $\delta \mathbf{u}$.
	In order to do this, define the adjoint variable $\boldsymbol{\lambda}$ that obeys the adjoint linear differential equation
	\begin{equation}
	\label{eqn:generatorModelAdjoint}
	\boxed{
		\left(-\frac{d}{dt} - \mathbf{K}_{\mathbf{0}}^T - \sum_{i=1}^{n_c}u_i(t) \mathbf{B}_{\mathbf{e}_i}^T \right)\boldsymbol{\lambda}(t) = \boldsymbol{\gamma}(t),
	}
	\end{equation}
	subject to the final condition $\boldsymbol{\lambda}(t_e)=\mathbf{0}$.
	This equation is easily solved backwards in time.
	Integration by parts and substitution of the linearized model and adjoint dynamics yields the desired relationship
	\begin{multline*}
		\langle \boldsymbol{\gamma},\ \delta \mathbf{z} \rangle = 
		\int_{t_0}^{t_e}\boldsymbol{\lambda}(t)^T\left(\frac{d}{dt} - \mathbf{K}_{\mathbf{0}} - \sum_{i=1}^{n_c}u_i(t) \mathbf{B}_{\mathbf{e}_i} \right)\delta \mathbf{z}(t)\ dt \\
		= \int_{t_0}^{t_e}\boldsymbol{\lambda}(t)^T
		\begin{bmatrix}
	        \mathbf{B}_{\mathbf{e}_1}\mathbf{z}(t) & \cdots & \mathbf{B}_{\mathbf{e}_{n_c}}\mathbf{z}(t)
	    \end{bmatrix} \delta \mathbf{u}(t)\ dt.
	\end{multline*}
	From this, the gradient with respect to the input $\mathbf{u}(t)$ at time $t$ is given explicitly by
	\begin{equation}
	\label{eqn:ObjectiveGradient}
	\boxed{
		\boldsymbol{\nabla}_{\mathbf{u}(t)}J = \begin{bmatrix}
		\mathbf{z}(t)^T\mathbf{B}_{\mathbf{e}_1}^T \\
		\vdots \\
		\mathbf{z}(t)^T\mathbf{B}_{\mathbf{e}_{n_c}}^T
		\end{bmatrix} \boldsymbol{\lambda}(t) + \boldsymbol{\rho}(t).
	}
	\end{equation}
	
	Having an efficient gradient expression at hand, we can now use a standard gradient-based optimization algorithm such as the quasi-Newton method of Broyden, Fletcher, Goldfarb, and Shanno (BFGS) \cite{fletcher2013newton}, which we will use for the numerical examples in the following section. An alternative approach is to ``first discretize then optimize'' approaches (as are popular in optimal control \cite{JMO05}); or, if a more efficient solution method is required, one could use a Newton-solver, as is explained in more detail in appendix~\ref{app:NewtonSolver}.

	\begin{remark}
		In order to solve Problem~\eqref{eq:MPC_Koopman} using the first-order accurate Koopman operator interpolation discussed in Sec.~\ref{subsec:PracticalApproximation}, we use a discretization according to Eq.~\eqref{eqn:Discrete_EDMD_ROM}:
		\begin{equation}\label{eq:MPC_Koopman_discrete} \tag{K-MPC\textsuperscript{d}}
			\begin{aligned}
				\min_{\bu \in \Ucal^{\ell}} &\sum_{i=0}^{\ell-1} \hat{L}_{i+1}(\bz_{i+1},\bu_{i+1}) \\
				\mbox{s.t.}\quad \bz_{i+1} &= \bK^{\Delta t}_{\mathbf{0}}\bz_i + \sum_{j=1}^{n_c} [\bu_{i}]_j \bB^{\Delta t}_{\mathbf{e}_j}\bz_i \\
				\bz_0 &= \boldsymbol{\psi}(\bx_0).
			\end{aligned}
		\end{equation}
		Here, the operators $\bK^{\Delta t}$, $\bB^{\Delta t}_{\mathbf{e}_j}$ can equivalently represent EDMD-based approximations of the Koopman operators over $\Delta t$ using \eqref{eqn:Discrete_EDMD_ROM}, or Euler integration of the generator approximations obtained using eqn.~\eqref{eqn:Continuous_EDMD_ROM} over the same time interval as shown in theorem \ref{thm:AffineKoopmanAndEulerIntegration}.
		A similar reformulation for the adjoint equation \eqref{eqn:generatorModelAdjoint} can be made, and a discretized version of the gradient \eqref{eqn:ObjectiveGradient} can be computed.
	\end{remark}
	
	\section{Numerical examples}
	\label{sec:MPC}
	
	In this section, we demonstrate the remarkable performance of our data-efficient MPC framework using the affine generator K-ROM as well as the first-order accurate Koopman operator version.
	
	\subsection{Forced Duffing Equation}
	\begin{figure}[b!]
		\centering
		\parbox[b]{0.48\textwidth}{\centering (a) \\ \includegraphics[width=.45\textwidth]{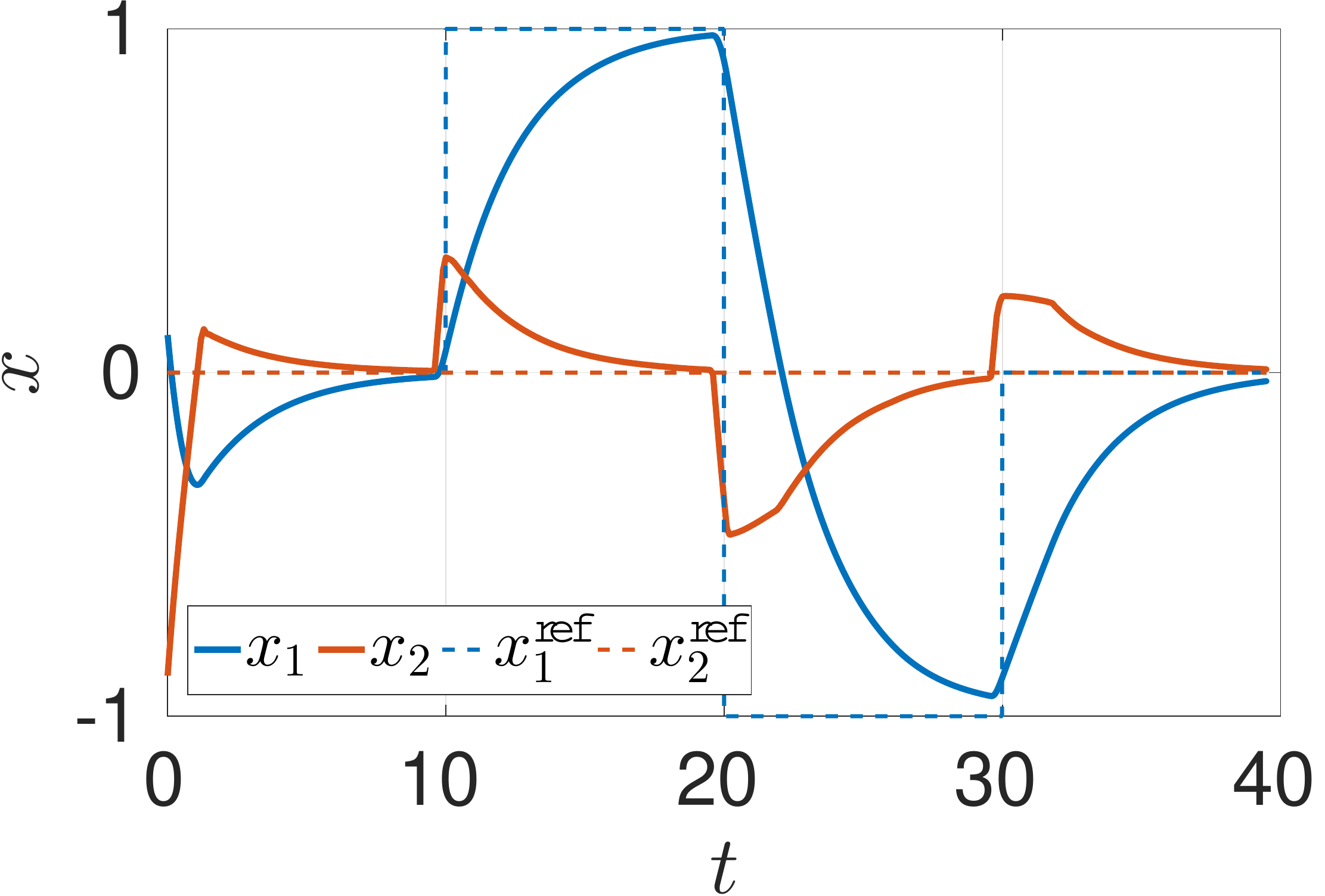}} \hfil
		\parbox[b]{0.48\textwidth}{\centering (b) \\ \includegraphics[width=.45\textwidth]{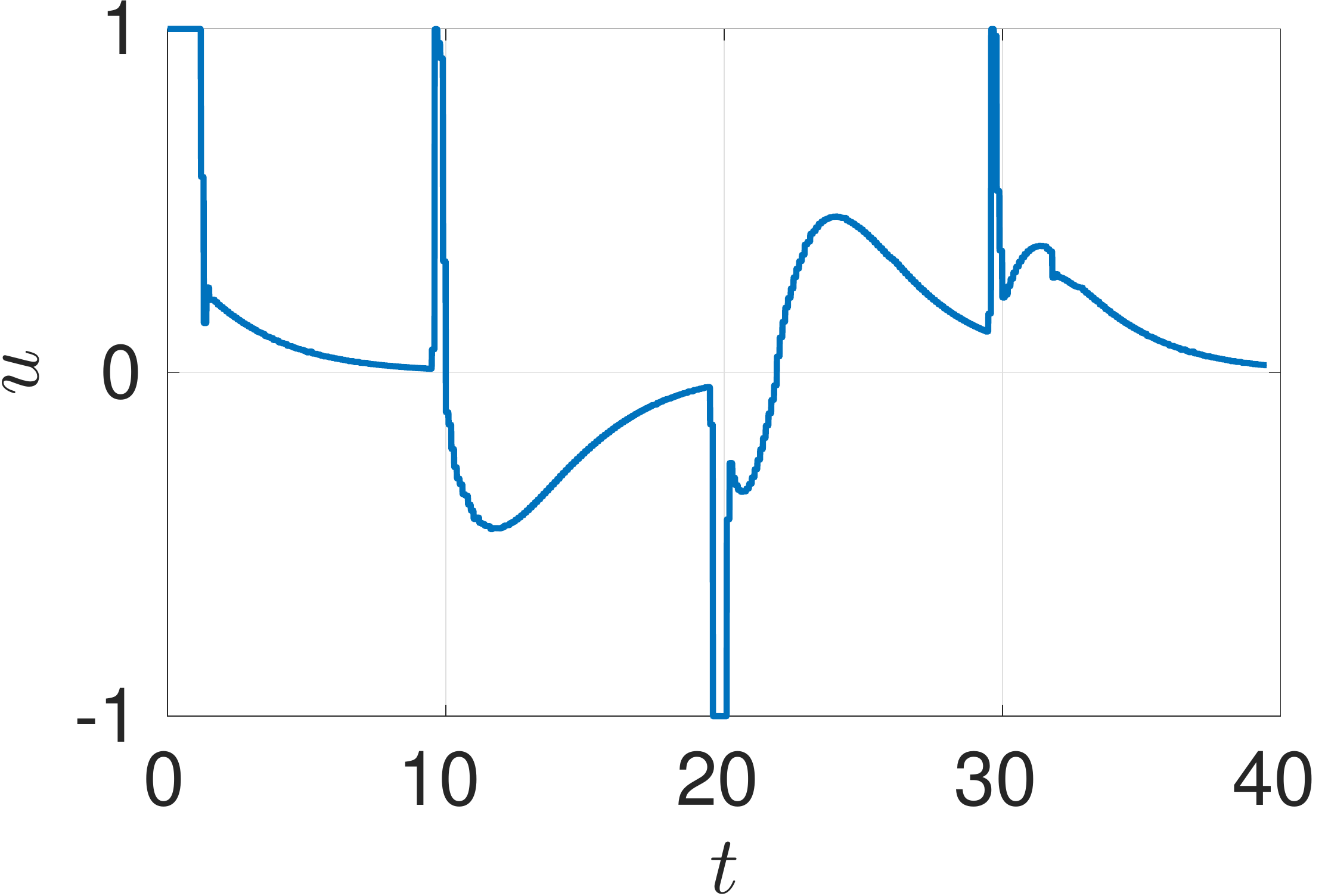}}
		\caption{Model predictive control of the Duffing equation. (a) Stabilization of three different stationary points. (b) The corresponding control input.}
		\label{fig:DuffingMPC}
	\end{figure}
	For the Duffing equation, we use the same numerical approximation for the generator as presented in Section~\ref{subsec:GenDuffing}. 
	We discretize Problem \eqref{eq:MPC_Koopman} by integrating the generator over a time step of $\Delta t = 0.1$ sec, which results in Problem \eqref{eq:MPC_Koopman_discrete}. 
	We then set $\ell = 5$ and prescribe a piecewise constant reference trajectory in order to stabilize the system at different points on the $x_1$-axis.
	The MPC problem is solved repeatedly until the final time of $40$ sec is reached.
	The result is shown in Fig.~\ref{fig:DuffingMPC}, where the stabilization of three different points is shown. 
	When comparing the performance to a controller using the nonlinear dynamics, the results are indistinguishable.
	
	\subsection{1D Burgers equation}
	As a second example, we consider the 1D Burgers equation with periodic boundary conditions and $\nu = 0.01$:
	\begin{align*}
		\dot{v}(t,\xi) - \nu \Delta v(t,\xi) + v(t,\xi) \nabla v(t,\xi) &= u(t) \chi(\xi).
	\end{align*}
	Here, $v$ denotes the state depending on space $\xi$ and time $t$, and the system is controlled by a shape function $\chi$ (see Fig.~\ref{fig:BurgersMPC} (a)), scaled by the input $u(t)\in [-0.025, 0.075]$. 
	Similar to the cylinder flow example, the data is collected from one trajectory with a piecewise constant input signal $u \in \{ -0.025, 0.075\}$. To realize a fast controller, we do not observe the full state, but only four points distributed equidistantly in space, cf.~Fig.~\ref{fig:BurgersMPC}~(a). 
	The observable $\psi$ then contains all monomial of these observations up to a a degree of two.
	Finally, we set the time step to $\Delta t = 0.5$ and choose a prediction horizon $\ell = 3$.
	
	The aim is to track a sinusoidal reference trajectory, i.e.,
	\[
		v^{\mathsf{ref}}(t, \xi) = 0.05 \sin(\pi t / 30) + 0.5 \quad \mbox{for}~\xi\in \{0,0.5,1,1.5\}.
	\]
	The solution of the K-ROM MPC is shown in Fig.~\ref{fig:BurgersMPC}, where the optimal input is shown in (b) and the optimal trajectory of the observable is shown in (c). 
	Finally, the corresponding trajectory of the full-state is shown in (d), and we see that we succeed in controlling the nonlinear Burgers equation using a 15-dimensional bilinear surrogate model which can be solved many orders faster than the full-state problem.
	
	\begin{figure}[h!]
		\centering
		\parbox[b]{0.48\textwidth}{\centering (a) \\ \includegraphics[width=.42\textwidth]{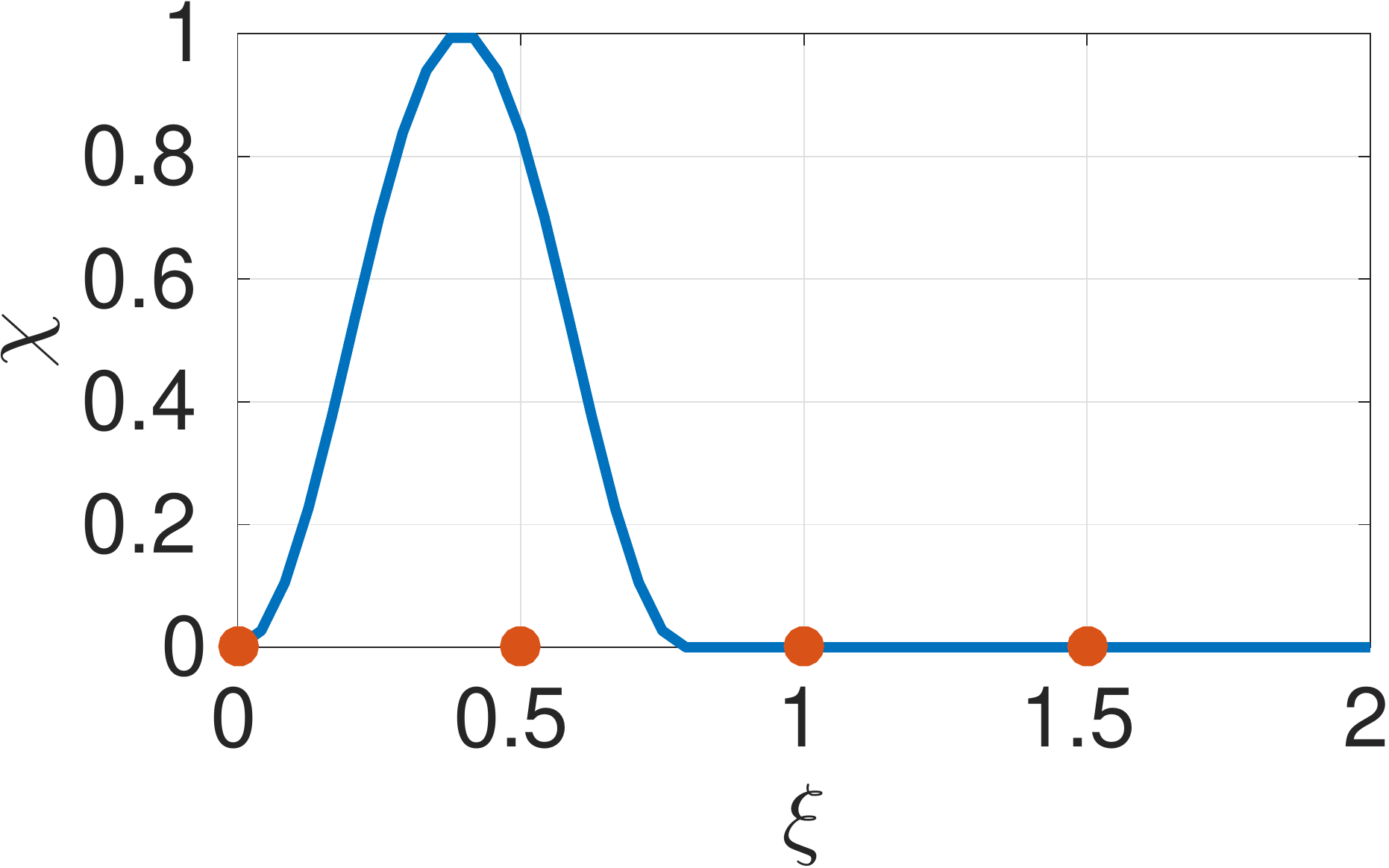}} \hfil
		\parbox[b]{0.48\textwidth}{\centering (b) \\ \includegraphics[width=.44\textwidth]{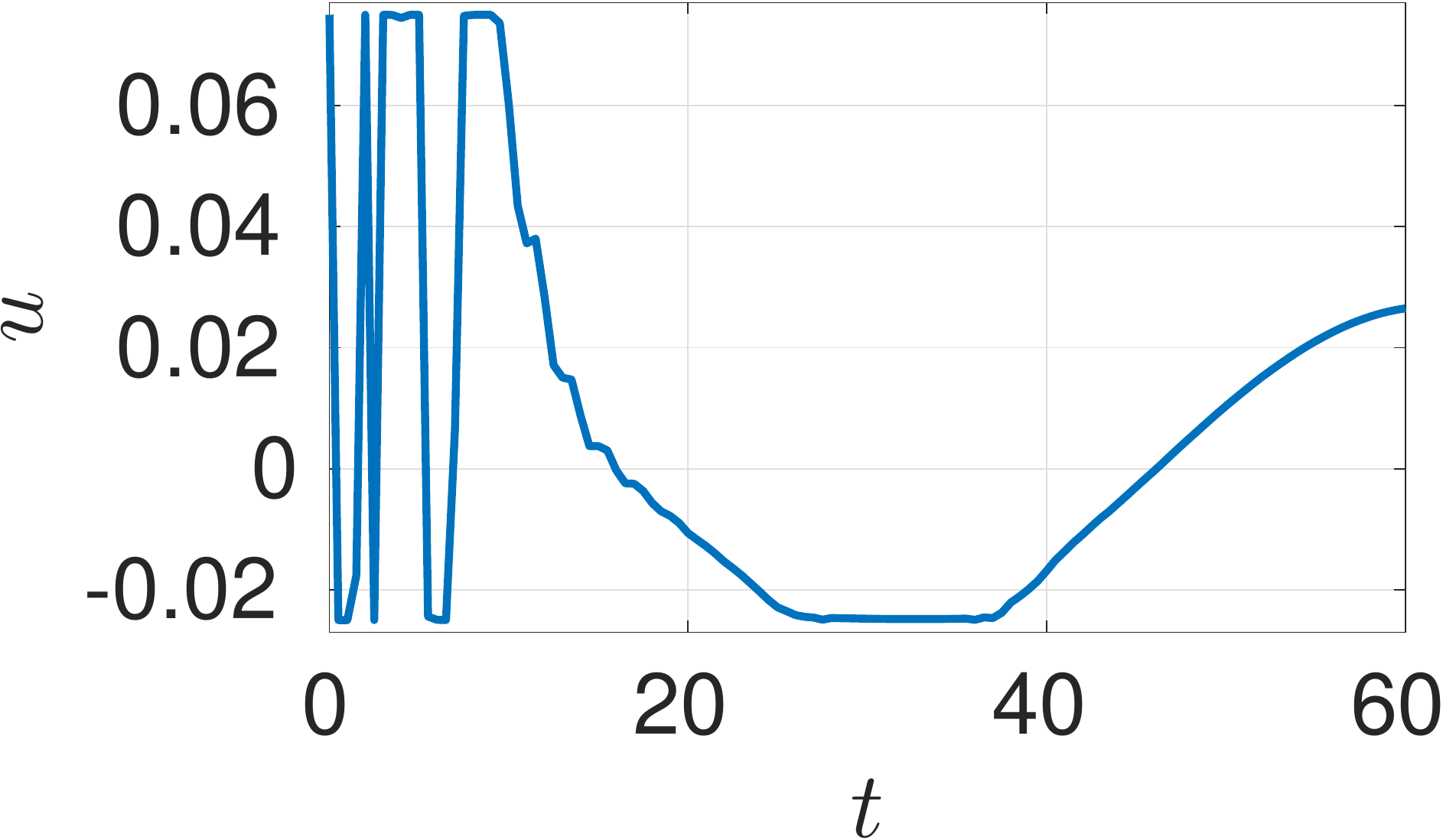}} \\
		\parbox[b]{0.48\textwidth}{\centering (c) \\ \includegraphics[width=.42\textwidth]{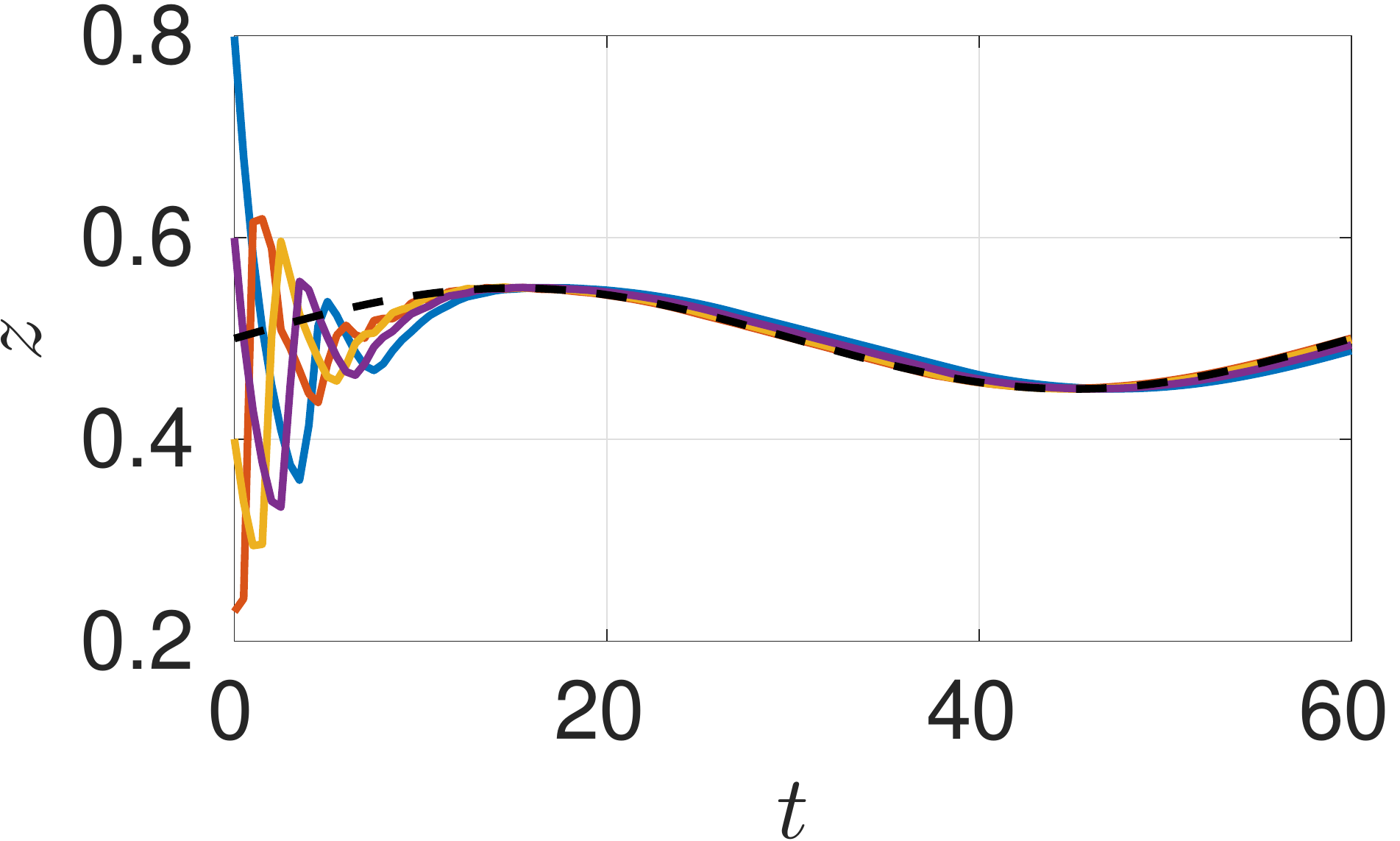}} \hfil
		\parbox[b]{0.48\textwidth}{\centering (d) \\ \includegraphics[width=.42\textwidth]{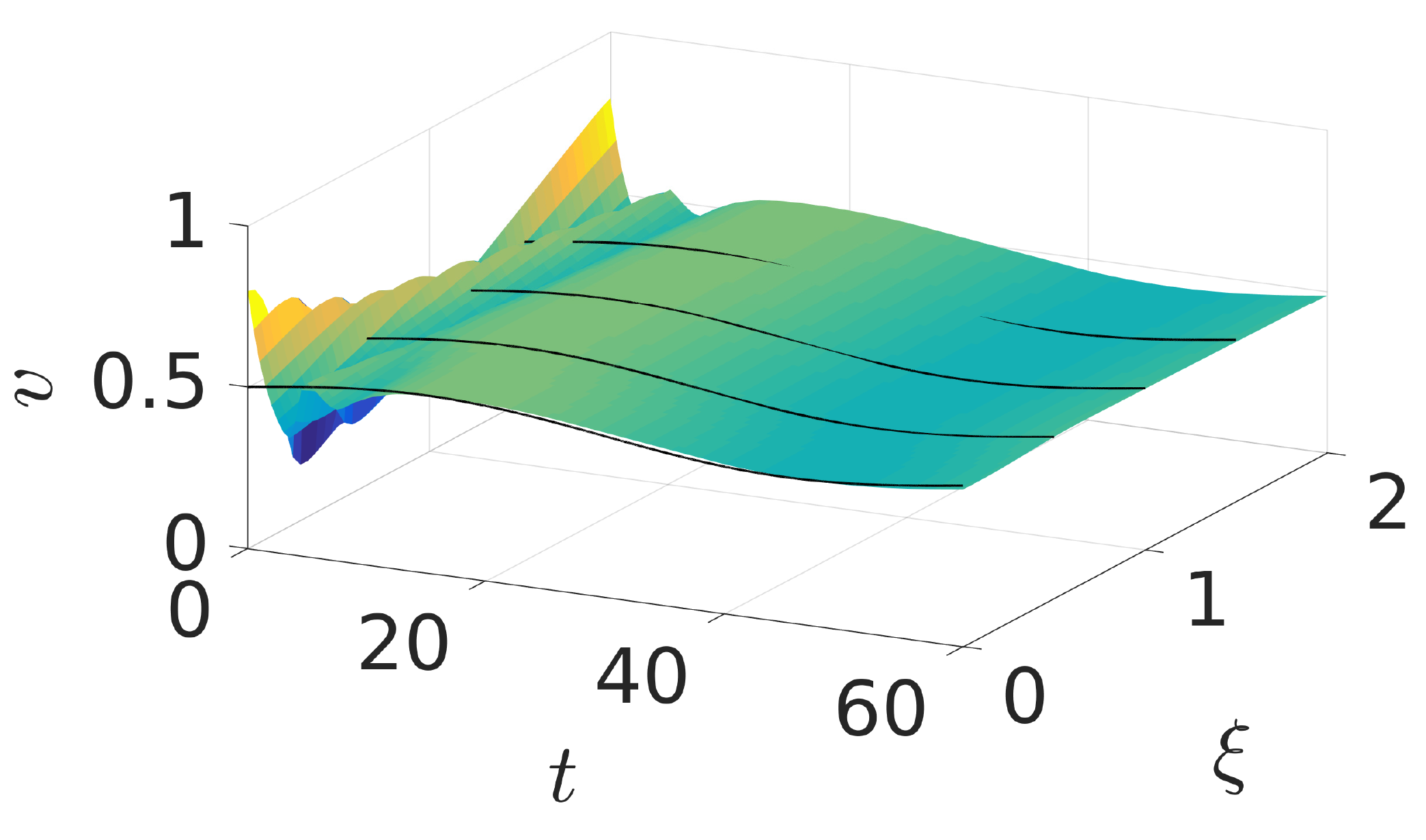}}
		\caption{Model predictive control of the 1D Burgers equation. (a) Shape function $\chi$ and the points at which the state is observed. (b) The optimal control input. (c) The optimal trajectory of the observed state as colored lines. The reference trajectory is identical for all four states. (d) The corresponding full-state solution.}
		\label{fig:BurgersMPC}
	\end{figure}
	
	\subsection{The fluidic pinball}
	As our final example, we use a flow control problem of high complexity. 
	In particular, we consider the \emph{fluidic pinball} \cite{DPMN18}, which shows chaotic behavior already at moderate Reynolds numbers. 
	The fluidic pinball is a configuration with three cylinders placed on the edges of an equidistant triangle with edge length $1.5 R$, where $R$ is the cylinder radius, cf.~Fig.~\ref{fig:fluidicPinball}. 
	Similar to the cylinder flow example in Section~\ref{subsec:PDE}, we use the first order approximation via the Koopman operator, as we do not have an explicit representation of the time derivative of the state. 
	In order to entirely avoid pointwise evaluations of the flow field, we only observe the forces acting on the cylinders. 
	To increase the number of observables, we use delay coordinates, which is a very common approach \cite{BBP+17,BPB+20}:
	\begin{equation*}
		\begin{aligned}
		\bz_i = \left(L^1_i, L^2_i, L^3_i, D^1_i, D^2_i, D^3_i, L^1_{i-1}, L^2_{i-1}, L^3_{i-1}, D^1_{i-1}, D^2_{i-1}, D^3_{i-1}\right),
		\end{aligned}
	\end{equation*}
	where $\bz_i$ is the observed quantity at time $t_i$, i.e., $\bz_i = \bz(t_i)$. 
	The goal is to control the lift forces of all three cylinders by rotating the cylinders two and three. 
	Hence, we simply have to track the corresponding entries of $z$ in the MPC problem:
	\begin{equation*}
		\begin{aligned}
		\min_{\bu \in [0,1]^{2\times \ell}} \sum_{i=0}^{\ell} &\left[\sum_{j=1}^{3}\left([\bz_i]_{j} - [\bz_i^{\mathsf{ref}}]_{j}\right)^2 + \gamma \|\bu_i\|_2 \right] \\
		\text{s.t.}\quad \bz_{i+1} = &\bK^{\Delta t}_{\bzero} \bz_{i} + \sum_{j=1}^{2} [\bu_i]_{j} \bB^{\Delta t}_{\mathbf{e}_j} \bz_{i} \quad \text{for}~i = 0,\ldots,\ell, \\
		\bz_0 = & \boldsymbol{\psi}(f(\bv_0, p_0)),
		\end{aligned}
	\end{equation*}
	where $\gamma$ is a weighting parameter.
	\begin{figure}
		\centering
		\includegraphics[width=.65\textwidth]{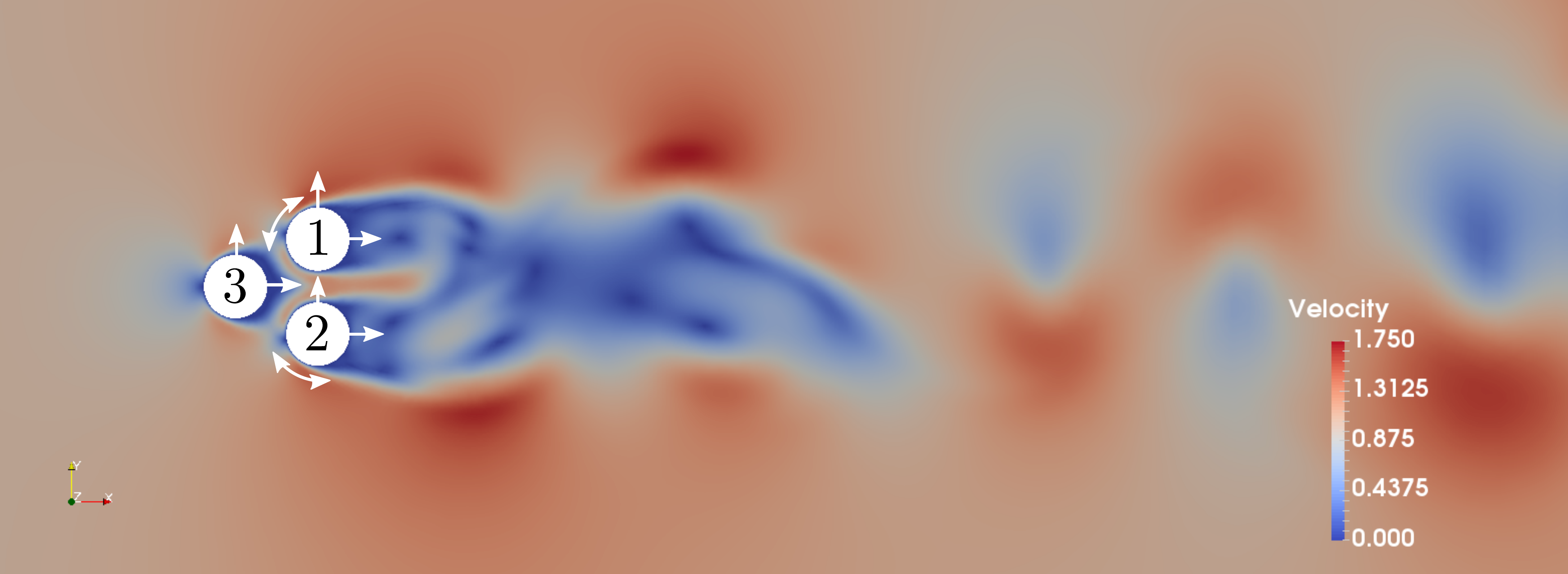}
		\caption{Setup of the fluidic pinball according to \cite{DPMN18}.}
		\label{fig:fluidicPinball}
	\end{figure}
	We want to allow control inputs between $-2$ and $2$ in each direction and since the control does not enter into the system in an affine manner, we create multiple K-ROMs which are \emph{localized\/} in the control domain. 
	This means that we approximate five Koopman operators $\Kcal_{\mathcal{I}}$ with 
	\begin{equation*}
	\begin{aligned}
	\mathcal{I} \in \left\{ 
	\left(\begin{array}{c} 0 \\ 0 \end{array}\right), 
	\left(\begin{array}{c} -2 \\ \color{white}-\color{black}0 \end{array}\right), 
	\left(\begin{array}{c} 2 \\ 0 \end{array}\right), 
	\left(\begin{array}{c} \color{white}-\color{black}0 \\ -2 \end{array}\right), 
	\left(\begin{array}{c} 0 \\ 2 \end{array}\right)
	\right\}
	\end{aligned}
	\end{equation*}
	and construct four reduced models of the form \eqref{eqn:Discrete_EDMD_ROM}, each of which is valid on the respective part of the control domain, cf.~Fig.~\ref{fig:LocalKROM}. 
	Similar ideas of localized reduced order models are often used in the \emph{reduced basis} community (cf., e.g., \cite{AHKO12,BDPV18}).
	\begin{figure}
		\centering
		\includegraphics[width=0.4\textwidth]{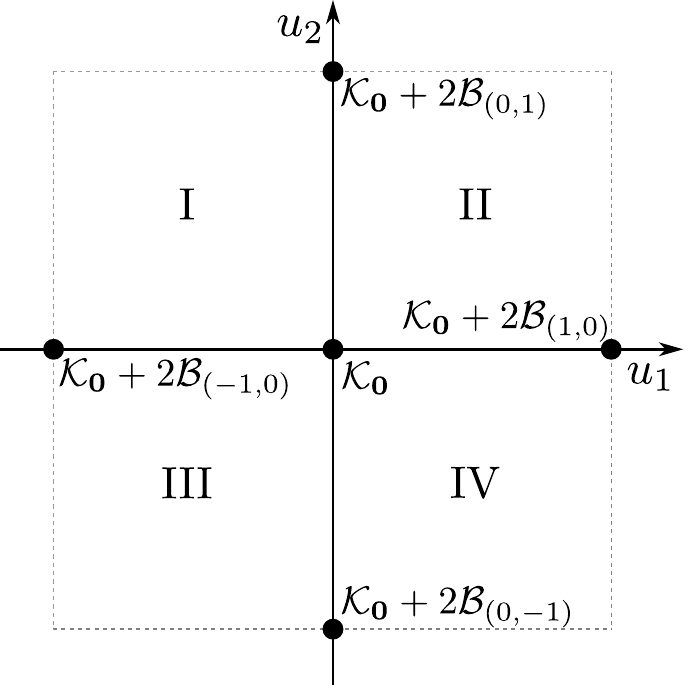}
		\caption{The four surrogate models I to IV are locally valid in the control domain $\Ucal$.}
		\label{fig:LocalKROM}
	\end{figure}
	
	We now study the performance of the behavior of the K-ROM-based controller for different Reynolds numbers. 
	In \cite{DPMN18} it was argued that the fluidic pinball without control possesses a quasi-periodic solution for $90 < Re < 120$, and that it behaves chaotically for $Re \geq 120$. 
	Consequently, it becomes more and more challenging to construct accurate surrogate models with increasing Reynolds number.
	
	As the first case, we set $Re=100$, i.e., we have quasi-periodic dynamics in the uncontrolled system. 
	In this case, a simple DMD approximation ($\psi(\bz) = \bz$) is sufficiently accurate, and we obtain a twelve-dimensional bilinear surrogate model with a time step of $\Delta t = 0.1$. 
	In comparison to solving the full system (where we use adaptive time stepping such that the CFL number is below 0.5), a speed-up factor of more than six orders of magnitude is obtained, and real-time applicability is achieved. 
	The resulting system behavior for a piecewise constant reference state is shown in Fig.~\ref{fig:fluidicPinball_Re100} and we observe excellent control performance.
	\begin{figure}
		\centering
		\parbox[b]{0.48\textwidth}{\centering (a) \\ \includegraphics[width=.48\textwidth]{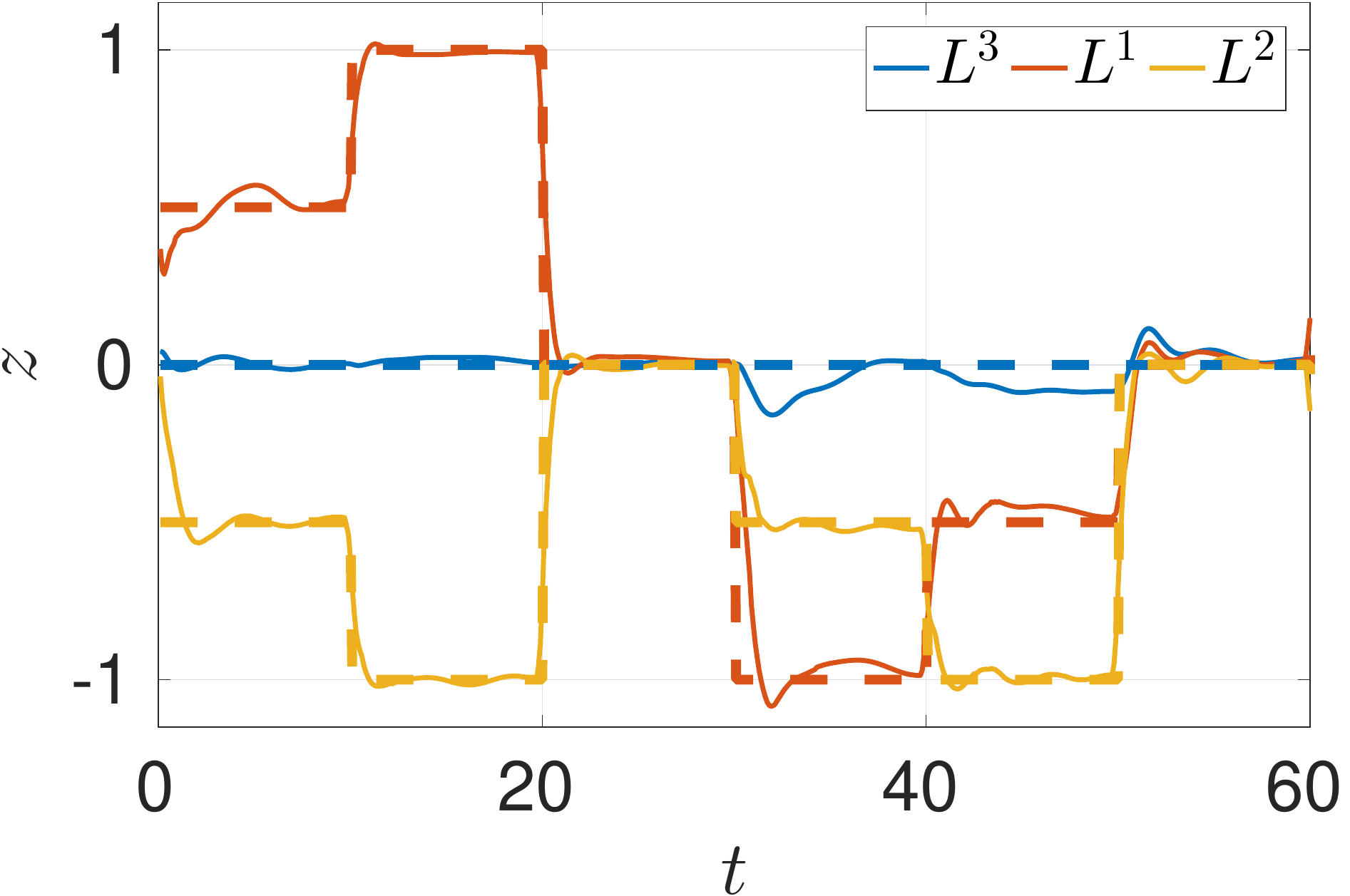}}
		\parbox[b]{0.48\textwidth}{\centering (b) \\ \includegraphics[width=.48\textwidth]{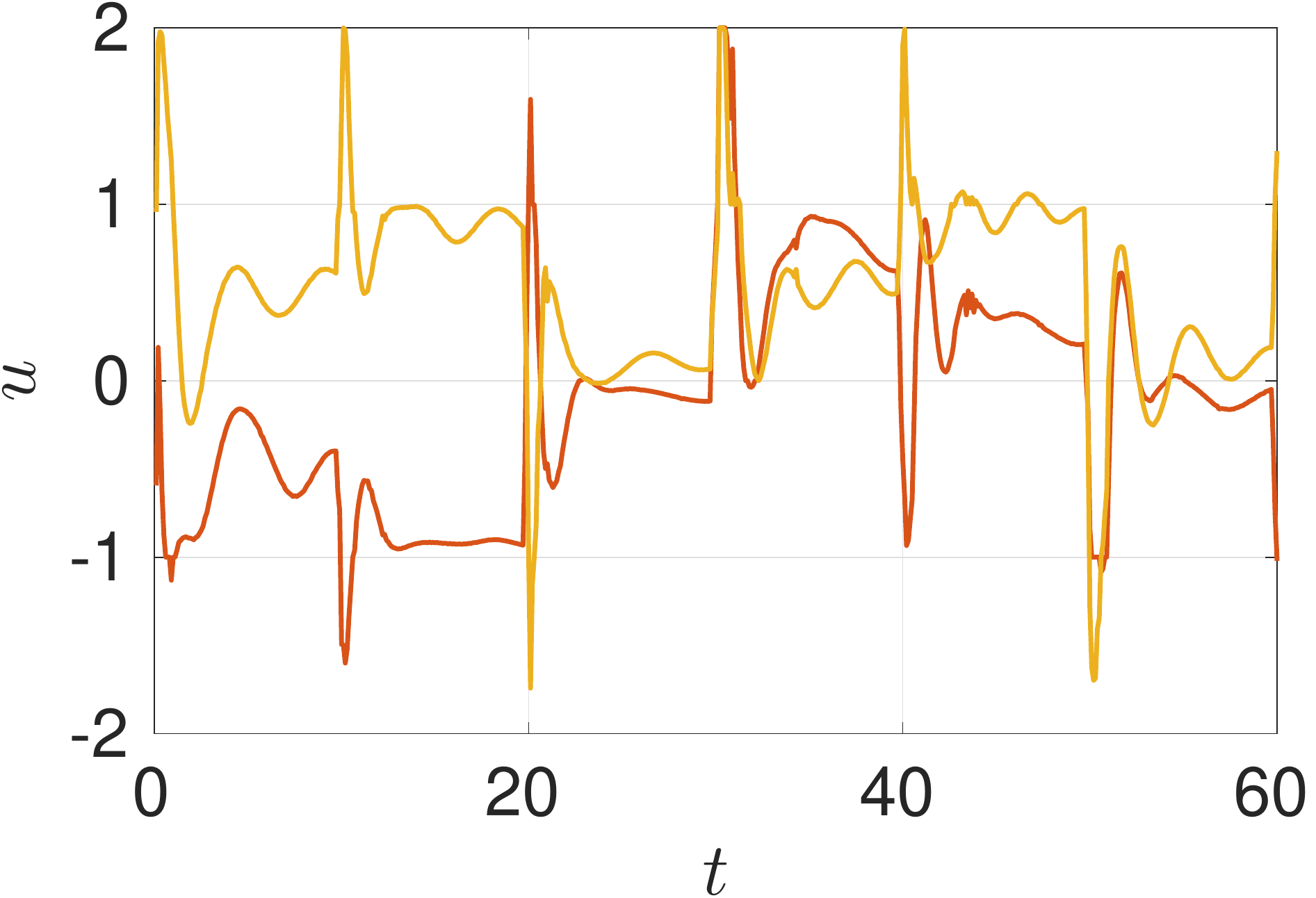}}
		\caption{(a) Lift and corresponding reference trajectories for the three lift forces at $Re = 100$, computed with the K-MPC approach with $\ell=5$. (b) The corresponding control inputs for the cylinders 1 and 2.}
		\label{fig:fluidicPinball_Re100}
	\end{figure}
	
	As the second case, we set $Re=140$ such that the system exhibits a mild
        form of chaos. 
	We see in Fig.~\ref{fig:fluidicPinball_Re140}~(a) that the fluctuations around the desired state are increased. 
	There are two possible reasons for the increased fluctuations:
	first, the system is chaotic and hence much more difficult to control; furthermore, the accuracy of the K-ROM is lower than for the quasi-periodic case such that the MPC problems \eqref{eq:MPC} and \eqref{eq:MPC_Koopman} do not necessarily possess the same solutions any longer. 
	Nevertheless, the control task is still performed quite satisfactorily, considering that we have replaced a nonlinear PDE with 
	$\approx150,000$ degrees of freedom by a twelve-dimensional linear system.
	\begin{figure}
		\centering
		\parbox[b]{0.48\textwidth}{\centering (a) \\ \includegraphics[width=.48\textwidth]{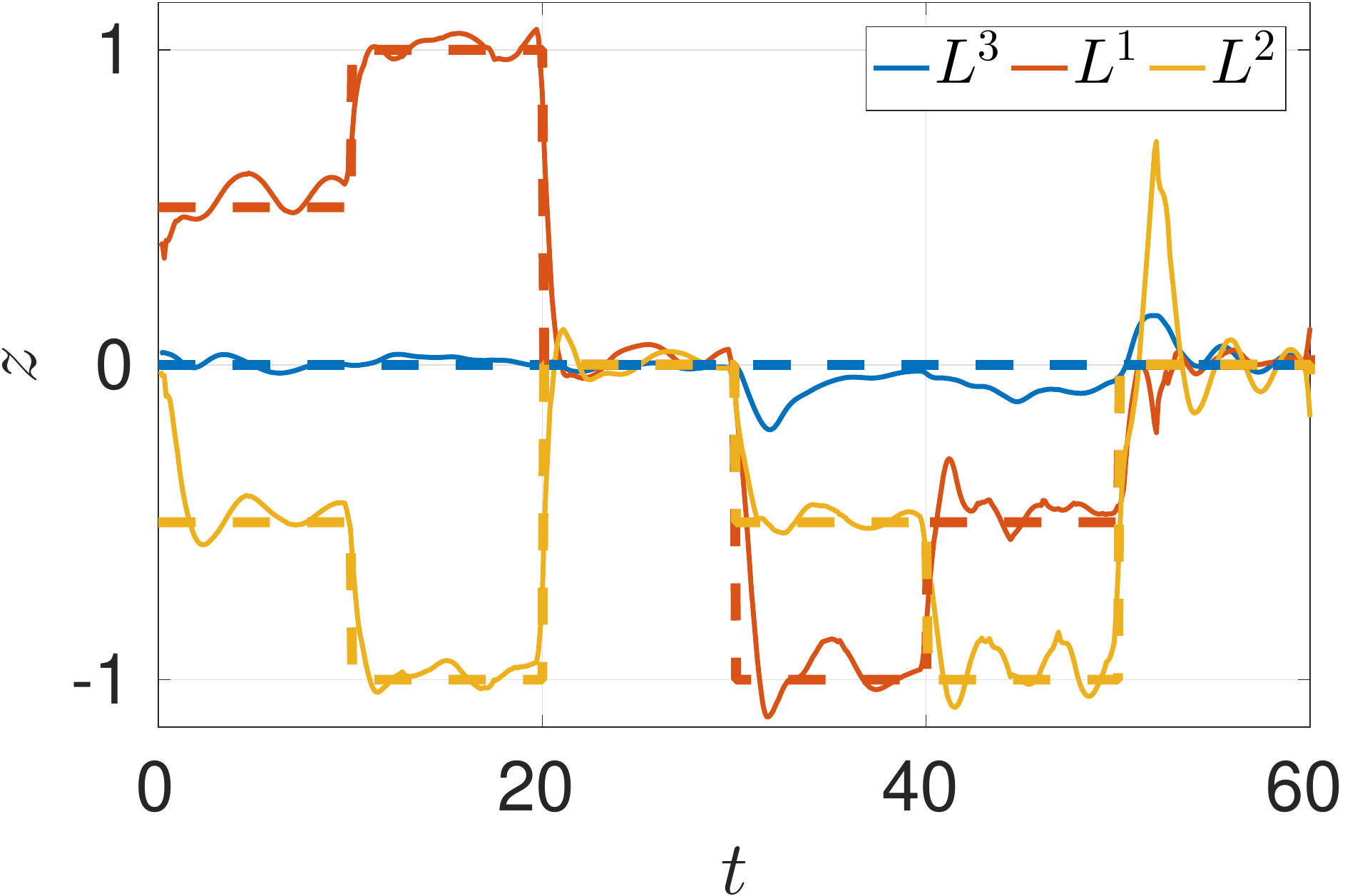}}
		\parbox[b]{0.48\textwidth}{\centering (b) \\ \includegraphics[width=.48\textwidth]{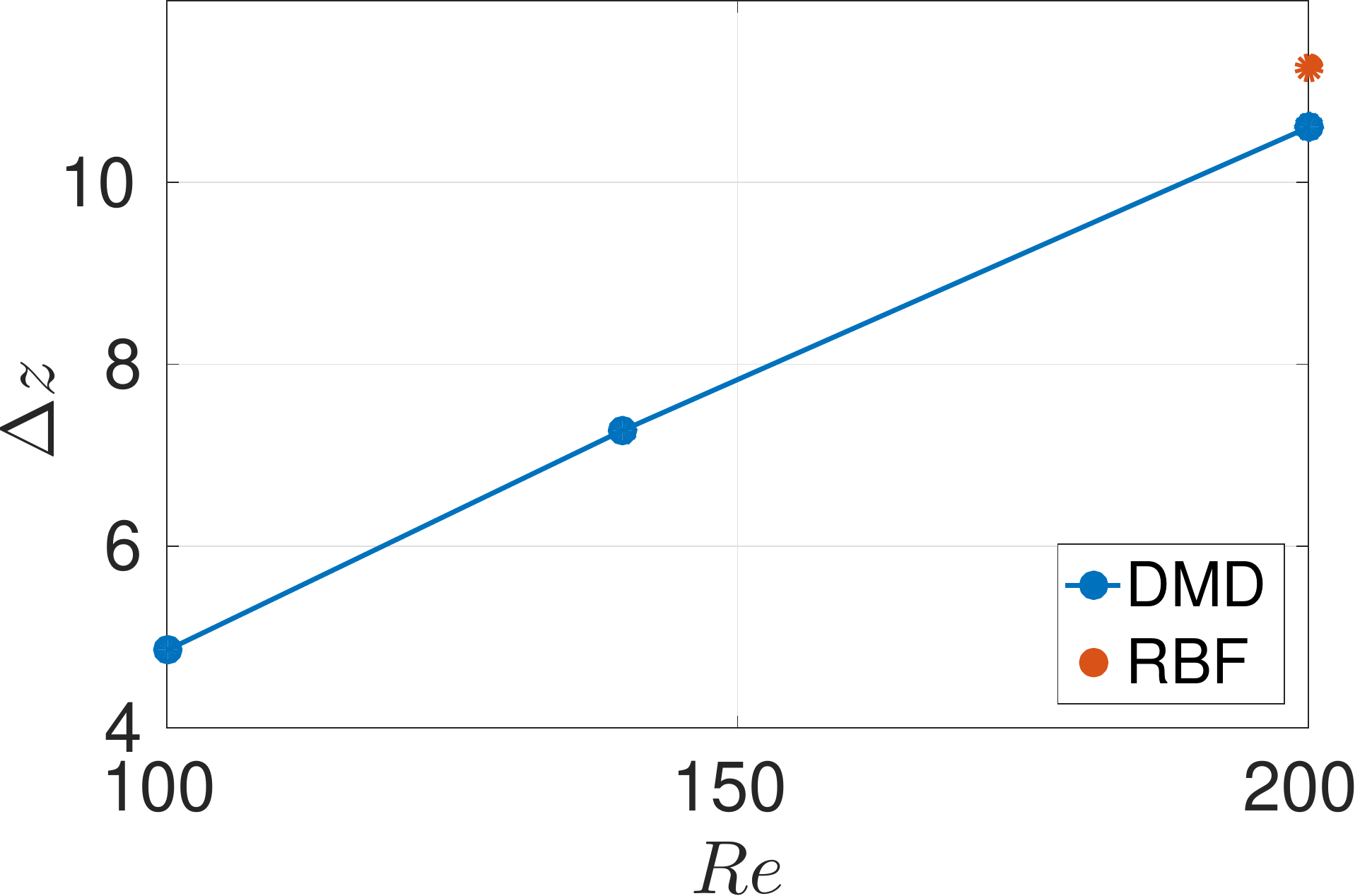}}
		\caption{(a) Similar to Fig.~\ref{fig:fluidicPinball_Re100} (a) but with $Re = 140$. (b) The control error $\Delta z$ with increasing Reynolds number.}
		\label{fig:fluidicPinball_Re140}
	\end{figure}
	
	When further increasing the Reynolds number well into the chaotic regime, we observe still stronger oscillations around the desired state; see Fig.~\ref{fig:fluidicPinball_Re140}~(b), where the tracking error $\Delta z = \int_{0}^{60}  \left \| z(t) - z^{\mathsf{ref}(t)}\right \|^2 \, dt$ is shown. 
	The increased oscillations are presumably due to the two reasons mentioned above: complexity of the dynamics and prediction quality. 
	In order to study the dependence of the prediction quality, we compare the approximation via DMD to one where we additionally place $1000$ radial basis functions in the 12-dimensional observation space as a Halton set (i.e., quasi-randomly). 
	The two runs are compared in Fig.~\ref{fig:fluidicPinball_Re200} and also in Fig.~\ref{fig:fluidicPinball_Re140}~(b). 
	We see that the much higher dimension of $\psi$ has no positive impact on the control performance, which indicates that the inaccuracy is mainly due to the chaotic dynamics.
	\begin{figure}
		\centering
		\parbox[b]{0.48\textwidth}{\centering (a) \\ \includegraphics[width=.48\textwidth]{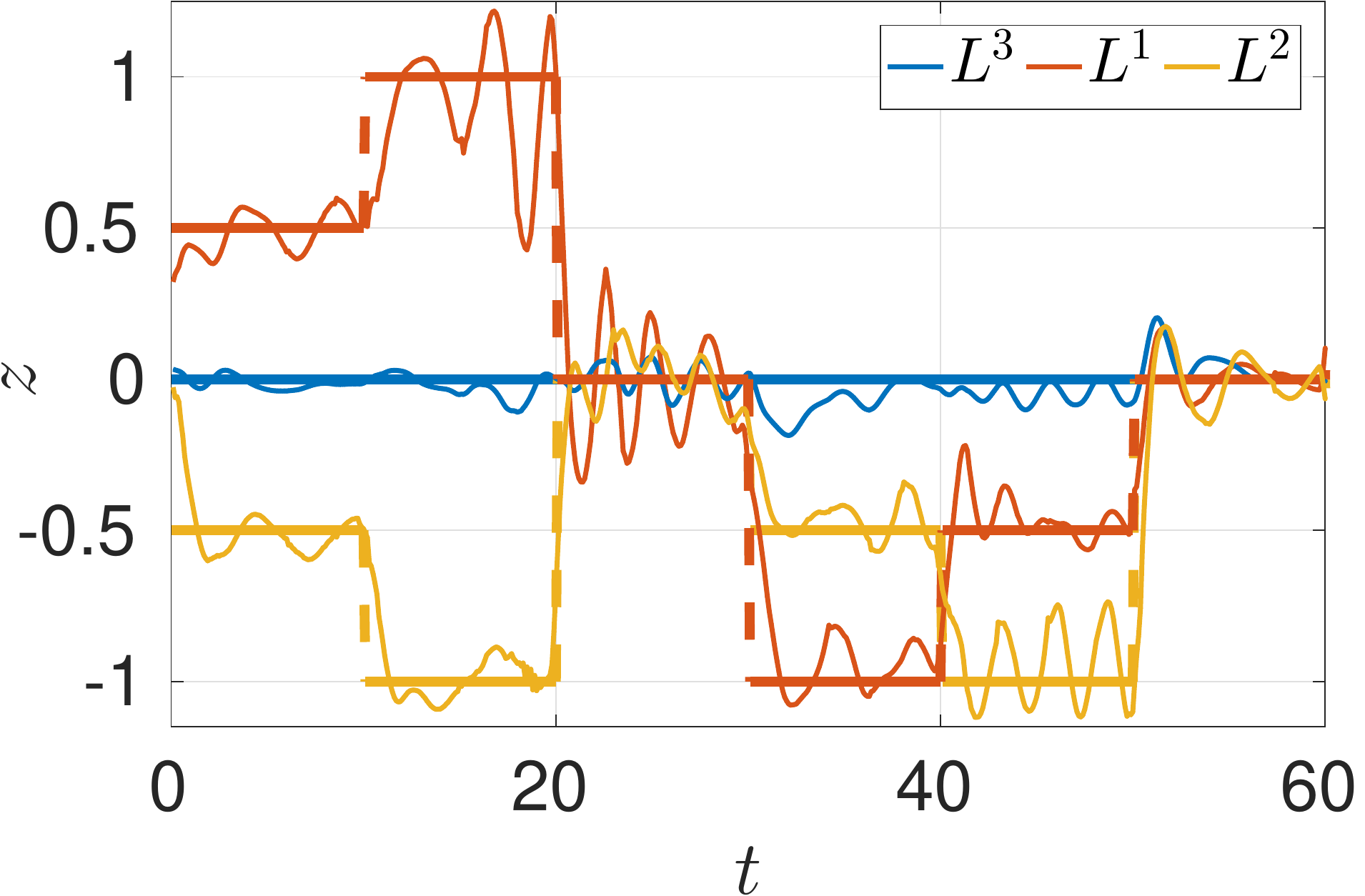}}
		\parbox[b]{0.48\textwidth}{\centering (b) \\ \includegraphics[width=.48\textwidth]{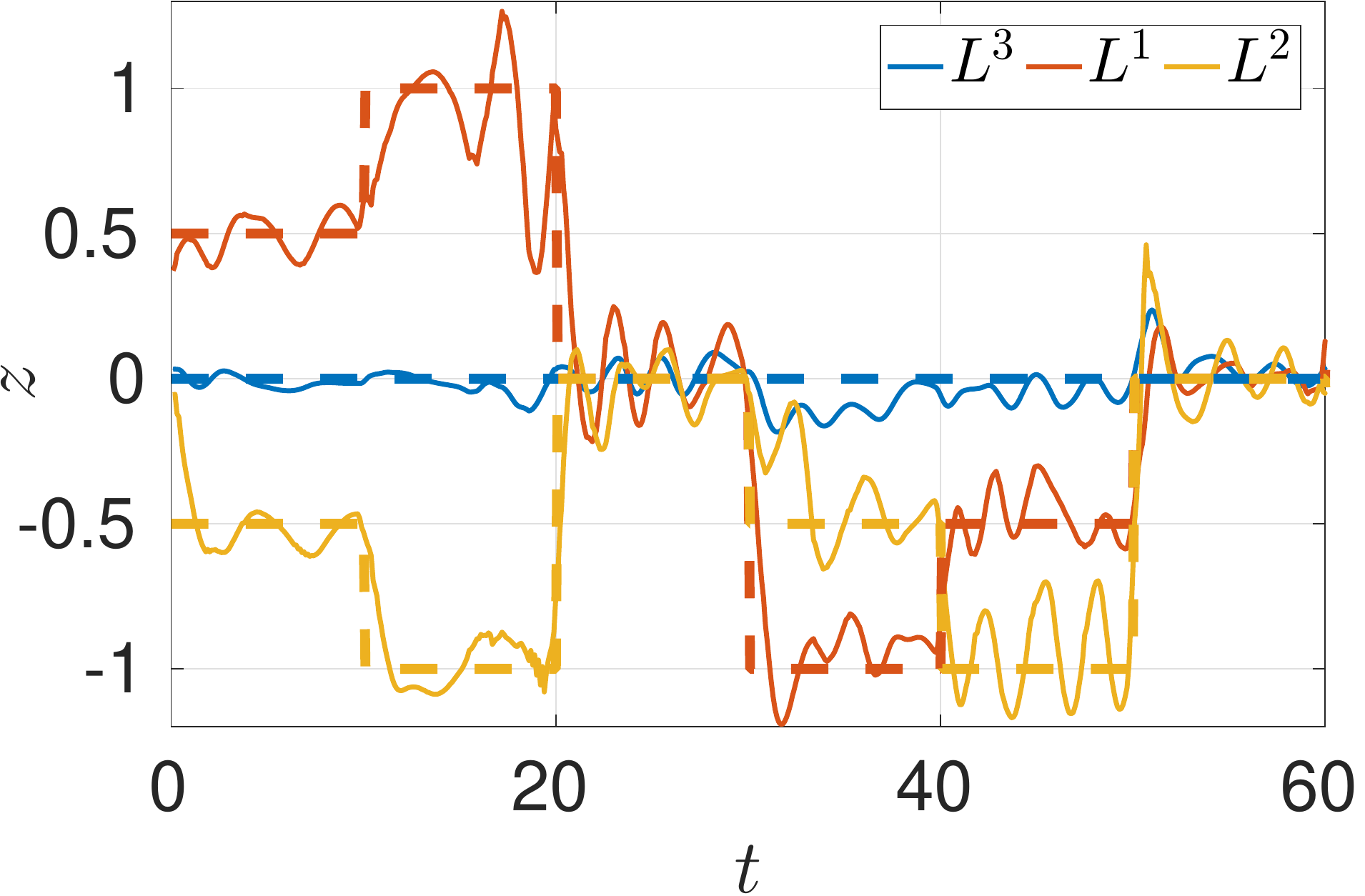}}
		\caption{Similar to Fig.~\ref{fig:fluidicPinball_Re100} (a) but with $Re = 200$. (a) Approximation via DMD as for the two previous cases. (b) Approximation via 1000 randomly distributed radial basis functions.}
		\label{fig:fluidicPinball_Re200}
	\end{figure}
	
	\section{Conclusion}
	\label{sec:Conclusion}
	
	We have presented a new approach for data-driven optimal control which is based on the Koopman generator. Multiple Koopman generators are approximated via EDMD at different constant control inputs, and intermediate control values can be approximated by linear interpolation between these operators, which yields a bilinear control system. 
	For control affine systems, the interpolation does not introduce an additional error. 
	Furthermore, when using the Koopman operator instead of the generator (which is often easier to compute numerically), the interpolation is accurate to first order in the timestep. Several examples show that the approach leads to very good results even in situations where this assumption does not hold. 
	
	Due to the larger step sizes and the linearity of the K-ROM, the reduced model can be solved significantly faster, in the case of the 2D Navier--Stokes equations by five to six orders of magnitude. An additional benefit is that since the model is bilinear, we can use efficient solution methods for the reduced control problem. Due to the restriction to several autonomous systems, the training data requirements are very low.
	
	One further direction of research is to develop stronger statements about the error for the K-ROM approach, e.g., concerning the basis size or the required data. From a control theoretic perspective, it would be very interesting to investigate whether the notion of controllability can be carried over to nonlinear systems. Moreover, feedback controllers specifically tailored to bilinear systems (using a time-dependent Riccatti equation, for instance) could be used instead of MPC to further increase the real-time capability. In terms of numerical efficiency, automated methods for choosing appropriate basis functions for the system dynamics (e.g., via dictionary learning \cite{LDBK17}) or Koopman approximations based on neural networks \cite{OR19} could help to further improve the prediction accuracy and consequently, the range of applicability.
	
	\bibliographystyle{unsrt}
	\bibliography{Bibliography}
	
	\appendix
	\section{Koopman-Kolmogorov Theorem Proof}
	\label{app:KoopmanKolmogorovProof}
	\begin{proof}
		Consider a trajectory $\bx(t)$ of the system starting from any initial condition $\bx_0\in\mathcal{X}$.
		The time derivative of an observable $\psi\in\Vcal$ along the trajectory $\bx(t)$ may be expressed in terms of the Koopman operator as
		\begin{equation*}
			\frac{d}{d t} \psi(\bx(t)) = 
			\frac{d}{d t} (\Kcal_{\bubar}^t\psi)(\bx_0) =
			\frac{d}{d t} \psi (\bPhi_{\bubar}^t(\bx_0)).
		\end{equation*}
		Since the flow is autonomous, the flow map obeys
		\begin{equation*}
		\bPhi_{\bubar}^{t+\Delta t}(\bx_0) = 
		\bPhi_{\bubar}^{\Delta t}(\bPhi_{\bubar}^t(\bx_0)),
		\end{equation*}
		allowing us to express the time derivative along the trajectory as
		\begin{equation*}
		\frac{d}{d t} (\Kcal_{\bubar}^t\psi)(\bx_0) =
		\lim_{\Delta t\to 0^+} \frac{\psi(\bPhi_{\bubar}^{\Delta t}(\bPhi_{\bubar}^t(\bx_0))) - \psi(\bPhi_{\bubar}^t(\bx_0))}{\Delta t}.
		\end{equation*}
		By definition of the Koopman generator, the right hand side of the above equation reduces to
		\begin{equation*}
		\frac{d}{d t} (\Kcal_{\bubar}^t\psi)(\bx_0) =
		(\Kcal_{\bubar}\psi)(\bPhi_{\bubar}^t(\bx_0)),
		\end{equation*}
		and by definition of the Koopman operator,
		\begin{equation*}
		\frac{d}{d t} (\Kcal_{\bubar}^t\psi)(\bx_0) = (\Kcal_{\bubar}^t \Kcal_{\bubar}\psi)(\bx_0).
		\end{equation*}
		This proves Eq.~\eqref{eqn:KoopmanOperatorODE} because $\psi\in\Vcal$ and $\bx_0\in\mathcal{X}$ were arbitrary.
		By direct substitution, one can verify that 
		\begin{equation*}
		(\Kcal_{\bubar}^t\psi)({\bx_0}) = \sum_{p=0}^{\infty}\frac{t^p}{p!}((\Kcal_{\bubar})^p\psi)(\bx_0)
		\end{equation*}
		satisfies the above ordinary differential equation, proving the validity of the expansion (\ref{eqn:KoopmanExponential}).
	\end{proof}
	
	\section{Efficient Newton-solver}
	\label{app:NewtonSolver}
	We here show that the special bilinear structure of the K-ROM enables an efficient Newton-type solver.
	For reasons that will become clear, let us assume that the objective function takes the quadratic form
	\begin{multline}
		\hat{L}(\mathbf{z}(t),\ \mathbf{u}(t),\ t) = (\mathbf{z}(t) - \mathbf{a}(t))^T\mathbf{Q}(t)(\mathbf{z}(t) - \mathbf{a}(t)) + \mathbf{u}(t)^T\mathbf{R}(t)\mathbf{u}(t),
	\end{multline}
	where $\mathbf{Q}(t)$ and $\mathbf{R}(t)$ are symmetric, positive definite matrices.
	This means that the functions
	\begin{equation*}
		\boldsymbol{\gamma}(t) = 2\mathbf{Q}(t)\left[\mathbf{z}(t) - \mathbf{a}(t)\right] \quad \mbox{and}\quad
		\boldsymbol{\rho}(t) = 2\mathbf{R}(t)\mathbf{u}(t)
	\end{equation*}
	are linear with respect to $\mathbf{z}(t)$ and $\mathbf{u}(t)$.
	
	Rather than minimizing the cost function by gradient descent, it is possible to solve for the input signal $\mathbf{u}\in\mathcal{F}([t_0,t_e0,\mathcal{U}])$ and resulting trajectory $\mathbf{z}$ that simultaneously satisfy the model dynamics and $\nabla_{\hat{u}_k}J = 0$ for all $k=1$, $\ldots$, $d$.
	This is done by simultaneously solving the dynamics eq.~\eqref{eqn:Continuous_EDMD_ROM} together with the adjoint equation eq.~\eqref{eqn:generatorModelAdjoint} and the optimality condition $\nabla_{\hat{u}_k}J = 0$ for all $k=1$, $\ldots$, $d$ with gradients found by eq.~\eqref{eqn:ObjectiveGradient} using a Newton method.
	
	Let us first restrict $\mathbf{u}$ to the subspace $\mathcal{F}([t_0,t_e0,\mathcal{U}])$ and define
	\begin{equation*}
	\mathbf{\hat{B}}_{k}(t) = \sum_{i=1}^{n_c}\varphi_{k,i}(t) \mathbf{B}_{\mathbf{e}_i}, \quad
	\mathbf{\hat{R}} = \int_{t_0}^{t_e}\boldsymbol{\varphi}(t)^T\mathbf{R}(t)\boldsymbol{\varphi}(t)\ dt,
	\end{equation*}
	yielding the transformation identities
	\begin{equation*}
	\sum_{i=1}^{n_c}u_i(t)\mathbf{B}_{\mathbf{e}_i} = 
	\sum_{k=1}^d\hat{u}_k \mathbf{\hat{B}}_{k}(t),
	\end{equation*}
	\begin{equation*}
	\int_{t_0}^{t_e}\mathbf{u}(t)^T\mathbf{R}(t)\mathbf{u}(t)\ dt = \mathbf{\hat{u}}^T\mathbf{\hat{R}}\mathbf{\hat{u}}
	\end{equation*}
	respectively.
	Using the Newton method, we want to solve for the discretized input $\mathbf{\hat{u}}\in\mathbb{R}^d$ and the signals $\mathbf{z}$ and $\boldsymbol{\lambda}$ that make the residuals
	\begin{equation}
	\label{eqn:DynamicsResidual}
	\mathbf{r}_{\mathbf{z}}(t) = \mathbf{\dot{z}}(t) - \left(\bK_{\bzero}+ \sum_{k=1}^d\hat{u}_k \mathbf{\hat{B}}_{k}(t)\right) \bz(t),
	\end{equation}
	\begin{equation}
	\label{eqn:AdjointResidual}
	\mathbf{r}_{\boldsymbol{\lambda}}(t) = 2\mathbf{Q}(t)\left[\mathbf{z}(t) - \mathbf{a}(t)\right] + \boldsymbol{\dot{\lambda}}(t) + \left(\mathbf{K}_{\mathbf{0}} + \sum_{k=1}^d\hat{u}_k \mathbf{\hat{B}}_{k}(t) \right)^T\boldsymbol{\lambda}(t),
	\end{equation}
	\begin{equation}
	\label{eqn:OptimalityResidual}
	\mathbf{r}_{\mathbf{\hat{u}}} = \int_{t_0}^{t_e} \boldsymbol{\varphi}(t)^T
	\begin{bmatrix}
	\mathbf{z}(t)^T\mathbf{B}_{\mathbf{e}_1}^T \\
	\vdots \\
	\mathbf{z}(t)^T\mathbf{B}_{\mathbf{e}_{n_c}}^T
	\end{bmatrix}\boldsymbol{\lambda}(t)\ dt
	+ 2\mathbf{\hat{R}}\mathbf{\hat{u}},
	\end{equation}
	equal to zero. Notice that each of the above equations is bilinear, making it very easy to compute the derivatives of the residuals with respect to the variables $\mathbf{\hat{u}}$, $\mathbf{z}$, and $\boldsymbol{\lambda}$.
	Each step of the Newton solver computes updates $\Delta \mathbf{z}$, $\Delta \boldsymbol{\lambda}$, and $\Delta \mathbf{\hat{u}}$ on the current iterates $\mathbf{z}$, $\boldsymbol{\lambda}$, and $\mathbf{\hat{u}}$ by solving the linear system
	\begin{equation*}
	\begin{bmatrix}
	\mathcal{D}_{\mathbf{z}}\mathbf{r}_{\mathbf{z}} & \mathbf{0} & \mathcal{D}_{\mathbf{\hat{u}}}\mathbf{r}_{\mathbf{z}} \\
	\mathcal{D}_{\mathbf{z}}\mathbf{r}_{\boldsymbol{\lambda}} & \mathcal{D}_{\boldsymbol{\lambda}}\mathbf{r}_{\boldsymbol{\lambda}} & \mathcal{D}_{\mathbf{\hat{u}}}\mathbf{r}_{\boldsymbol{\lambda}} \\
	\mathcal{D}_{\mathbf{z}}\mathbf{r}_{\mathbf{\hat{u}}} & \mathcal{D}_{\boldsymbol{\lambda}}\mathbf{r}_{\mathbf{\hat{u}}} & \mathcal{D}_{\mathbf{\hat{u}}}\mathbf{r}_{\mathbf{\hat{u}}}
	\end{bmatrix}_{(\mathbf{z},\boldsymbol{\lambda},\mathbf{\hat{u}})}
	\begin{bmatrix}
	\Delta \mathbf{z} \\
	\Delta \boldsymbol{\lambda} \\
	\Delta \mathbf{\hat{u}}
	\end{bmatrix} = -
	\begin{bmatrix}
	\mathbf{r}_{\mathbf{z}} \\
	\mathbf{r}_{\boldsymbol{\lambda}} \\
	\mathbf{r}_{\mathbf{\hat{u}}}
	\end{bmatrix}_{(\mathbf{z},\boldsymbol{\lambda},\mathbf{\hat{u}})}.
	\label{eqn:NewtonSystem}
	\end{equation*}
	The iterates are updated according to
	\begin{equation*}
	(\mathbf{z},\ \boldsymbol{\lambda},\ \mathbf{\hat{u}}) \leftarrow (\mathbf{z},\ \boldsymbol{\lambda},\ \mathbf{\hat{u}}) + (\Delta\mathbf{z},\ \Delta\boldsymbol{\lambda},\ \Delta\mathbf{\hat{u}})
	\end{equation*}
	during each step until convergence.
	
	We recommend using an iterative Krylov subspace algorithm like the Generalized Minimal Resisual Method (GMRES) to solve the system eq.~\eqref{eqn:NewtonSystem}. 
	This is because it is very easy to compute operator-vector products like $\begin{bmatrix}
	\mathcal{D}_{\mathbf{z}}\mathbf{r}_{\mathbf{\hat{u}}} & \mathcal{D}_{\boldsymbol{\lambda}}\mathbf{r}_{\mathbf{\hat{u}}} & \mathcal{D}_{\mathbf{\hat{u}}}\mathbf{r}_{\mathbf{\hat{u}}} \end{bmatrix} \mathbf{h}$, for example, by applying linearized versions of the residual equations \eqref{eqn:DynamicsResidual}, \eqref{eqn:AdjointResidual}, \eqref{eqn:OptimalityResidual}, directly to $\mathbf{h}$, rather than discretizing and storing the derivative operators as large matrices.
	
	\begin{remark}[The adjoints of bilinear models have favorable properties]
		If the model Eq.~\eqref{eqn:Continuous_EDMD_ROM} had general nonlinearities, then the Jacobian matrices would appear in the adjoint equation Eq.~\eqref{eqn:generatorModelAdjoint}.
		In the general case, these Jacobian matrices depend on $\mathbf{z}(t)$ and $\mathbf{u}(t)$ so that Eq.~\eqref{eqn:AdjointResidual} is no longer bilinear. 
		The Newton method would then require differentiating the Jacobian matrices appearing in Eq.~\eqref{eqn:AdjointResidual} with respect to $\mathbf{z}(t)$ and $\mathbf{u}(t)$ to find Hessians of the model dynamics.
		These additional terms make solving the optimization problem more costly, which could pose a problem for MPC applications.
	\end{remark}
	
\end{document}